\documentclass[a4paper,11pt] {article}

\usepackage{geometry}
\usepackage{mdwlist}
\usepackage[latin1]{inputenc}
\usepackage[T1]{fontenc}
\usepackage[latin1]{inputenc}
\usepackage[T1]{fontenc}
\usepackage[english]{babel}
\usepackage{fancybox}
\usepackage{xcolor}
\usepackage{bm}
\usepackage{tikz}
\usepackage{cite}
\usepackage{amsthm}

\newtheorem{rmq}{Remark}

\newtheorem{df}{Definition}
\newtheorem{thm}{Theorem}
\newtheorem{prop}{Proposition}
\newtheorem{lm}{Lemma}

\usepackage{amsmath}
\usepackage{amssymb}
\usepackage{graphicx}
\usepackage{wrapfig}
\usepackage{titlesec}
\titleformat\section{}{}{0pt}{\Large\scshape\bfseries\filcenter\thesection{} - }

\newcommand{\tn}[1]{\mathbb{#1}}
\newcommand{\Ov}[1]{\overline{#1}}
\newcommand{\vr}{\varrho}
\newcommand{\R}{\mathbb{R}}
\newcommand{\N}{\mathbb{N}}

\newcommand{\QT}{(0,T)\times \Omega}
\newcommand{\QTb}{[0,T]\times \overline{\Omega}}
\newcommand{\bu}{\bm{u}}
\newcommand{\bU}{\bm{U}}
\newcommand{\bn}{\bm{n}}

\newcommand{\bV}{\bm{V}}

\newcommand{\dxdt}{\, \dx\,\dt}
\newcommand{\dx}{{\rm d} {x}}
\newcommand{\dt}{{\rm d} t }
\newcommand{\dtau}{{\rm d} \tau }
\DeclareMathOperator{\dv}{div}

\renewcommand{\d}{\partial}

\DeclareMathOperator{\A}{{\cal{A}}}

\DeclareMathOperator{\E}{{\cal{E}}}
\DeclareMathOperator{\Rr}{{\cal{R}}}

\DeclareMathOperator{\D}{{\cal{D}}}
\numberwithin{equation}{section}
\title{Existence of weak solutions for compressible Navier-Stokes equations with entropy transport}
\author{David Maltese$^1$ \and Martin Mich\'alek$^{2,4}$ \and Piotr B. Mucha$^3$ \and
Antonin Novotn\'y$^1$  \and Milan Pokorn\'y$^2$ \and Ewelina Zatorska$^{3,5}$ }
\date{   }

\begin{document}

\maketitle

\noindent
\let\thefootnote\relax\footnote{{\hspace{-0.6cm}1. Institut de Math\'ematiques de Toulon, EA 2134, BP20132, 83957 La Garde, France. \\
2. Faculty of Mathematics and Physics, Charles University in Prague, Ke Karlovu 3, 121 16 Praha 2, Czech Republic. \\
3. Institute of Applied Mathematics and Mechanics, University of Warsaw, ul. Banacha 2, 02-097 Warsaw, Poland.\\
4. Institute of Mathematics of the Czech Academy of Sciences, \v{Z}itn\'{a} 25, 110 00 Praha, Czech Republic.
\\
5. Department of Mathematics, Imperial College London, 180 Queen's Gate,  London SW7 2AZ, United Kingdom.}}

{\bf Abstract.} We consider the compressible Navier-Stokes system with variable entropy.
The pressure is a nonlinear function of the density and the entropy/potential temperature which, unlike in the Navier-Stokes-Fourier system, satisfies only the transport equation. We provide existence results within three  alternative weak formulations of the corresponding classical problem. Our constructions hold for the optimal range of the adiabatic coefficients from the point of view of the nowadays existence theory.

\section{Introduction}
The purpose of this paper is to analyze the model of flow of compressible  viscous fluid with variable entropy. Such flow can be described by the compressible Navier-Stokes equations coupled with an additional equation describing the evolution of the entropy. In case when the conductivity is neglected, the changes of the entropy are solely due to the transport and the whole system can be written as:
\begin{subequations}\label{1}
\begin{equation}\label{cont1_1}
\partial_t \vr + \dv( \vr \bu) = 0~\text{in}~\QT,
\end{equation}
\begin{equation}\label{content1_1}
\partial_t (\vr s) + \dv (\vr s \bu ) =0~\text{in}~\QT,
\end{equation}
\begin{equation}\label{mom1_1}
\partial_t( \vr \bu) + \dv( \vr \bu \otimes \bu) + \nabla p  = \dv{\tn{S}}  ~\text{in}~\QT,
\end{equation}
\end{subequations}
where the unknowns are the density $\vr \colon (0,T) \times \Omega \to {\R_+}\cup \{0\}$, the entropy $s\colon (0,T) \times \Omega \to \R_+$ and the velocity of fluid $\bu\colon \QT \to \R^3$, and where $\Omega$ is  a three dimensional domain with a smooth boundary $\d \Omega$.

The momentum, the continuity and the entropy equations are additionally coupled by the form of the pressure $p$, we assume that
\begin{equation} \label{4}
p(\vr,s) = \vr^\gamma {\cal T}(s), \quad \gamma>1,
\end{equation}
where ${\cal T}(\cdot)$ is a given smooth and strictly monotone function from $\R_+$ to $\R_+$,  in particular 
${\cal T}(s)>0$ for $s>0$.

We assume that the fluid is Newtonian and that the viscous part of the stress tensor is of the following form
$$
\tn {S} = \tn{S}(\nabla \bu)=  2\mu \Big(\tn{D}(\bu) -\frac 13 \dv\bu \tn{I}\Big)  +\eta  \dv_x \bu \tn{I}
$$
with $\tn{D}(\bu)=\frac 12 (\nabla \bu +\nabla \bu ^T)$.
Viscosity coefficients $\mu$ and $\eta$ are  assumed to be constant, hence we can write
$$
\dv{\tn{S}}(\nabla \bu) = \mu \Delta \bu + (\mu + \lambda) \nabla \dv \bu
$$
with $\lambda = \eta -\frac 23 \mu$. To keep the ellipticity of the Lam\'e operator we require that
\begin{equation}\label{viscondition}
\mu>0, \quad 3\lambda + 2\mu > 0.
\end{equation}

The system is supplemented by the initial and the boundary conditions:
\begin{equation}\label{2}
\vr(0,x)= \vr_0(x),~ (\vr s)(0,x)=S_0(x),~ (\vr\bu)(0,x) = \bm{q}_0(x),
\end{equation}
\begin{equation}\label{3}
\bu_{| (0,T) \times \partial \Omega} = \bm{0}.
\end{equation}

System (\ref{1}) is a model of motion of compressible viscous gas with variable entropy transported by the flow. The quantity $\theta=[{\cal T}(s)]^{1/\gamma}$ can be also interpreted as a potential
temperature in which case the pressure (\ref{4}) takes the form $(\rho\theta)^\gamma$ and has been studied in \cite{FEKLNOZA,LM}.

We aim at proving the existence of global in time weak solutions to system  (\ref{1}). Note that at least for smooth solution the continuity equation (\ref{cont1_1}) allows us to reformulate  (\ref{content1_1}) as a pure transport equation for $s$, we have
\begin{subequations}\label{1e}
\begin{equation}\label{conte_1}
\partial_t \vr + \dv( \vr \bu) = 0~\text{in}~\QT,
\end{equation}
\begin{equation}\label{ent_trans}
\partial_t s + \bm{u} \cdot \nabla s = 0~\text{in}~\QT,
\end{equation}
\begin{equation}\label{mome_1}
\partial_t( \vr \bu) + \dv( \vr \bu \otimes \bu) + \nabla p  = \dv{\tn{S}}  ~\text{in}~\QT.
\end{equation}
\end{subequations}
In contrast to entropy equation in  system (\ref{1}) the above form is insensitive to appearance of vacuum states; in fact it is completely decoupled from the continuity equation. The regularity of the density  in the compressible Navier-Stokes-type systems is in general rather delicate matter. Therefore, one can expect that proving the existence of solutions to system (\ref{1}) requires more severe assumptions than to get a relevant solution to (\ref{1e}).
This observation will be reflected in the range of parameter $\gamma$ which determines the quality of {a priori} estimates for the argument of the pressure -- $Z=\vr[{\cal T}(s)]^{1\over\gamma}$ according to the notation from above.

In order to clarify this issue a little more let us introduce a third formulation of system (\ref{1}) describing the evolution of the pressure argument $Z= \vr [{\cal T}(s)]^{\frac 1\gamma}$ instead of the entropy itself. We have:
\begin{subequations}\label{2a}
\begin{equation}\label{cont1}
\partial_t \vr + \dv( \vr \bu) = 0~\text{in}~\QT,
\end{equation}
\begin{equation}\label{content1}
\partial_t Z + \dv(Z \bu) =0~\text{in}~\QT,
\end{equation}
\begin{equation}\label{mom1}
\partial_t( \vr \bu) + \dv( \vr \bu \otimes \bu) + \nabla Z^\gamma   = \dv{\tn{S}}~\text{in}~\QT.
\end{equation}
\end{subequations}
Again,  the above formulation is equivalent with the previous ones provided the solution is regular enough, which, however,  may not be true in case of weak solutions.

 The above discussion motivates distinction between the cases when the evolution of the entropy  is described by the continuity, the transport or the renormalized transport equation. Indeed, the form of the entropy equation, although used to describe the same phenomena, is a diagnostic marker indicating the notion of plausible solution to the whole system.
Our paper contains an existence analysis for all three systems: \eqref{1}, \eqref{1e} and \eqref{2a} within suitably adjusted
definitions of weak solutions. Such an approach allows us to emphasise the implications between the solutions and to better understand the restrictions of renormalization technique. These issues, absent in the analysis of the standard single density systems, are of great importance for more complex multi-component or multi-phase flows. Our results show possible applications of nowadays classical tools in the analysis of the Navier-Stokes system to challenging problems, 
e.g. constitutive equation involving nonlinear combinations of hyperbolic quantities: densities, concentrations, etc.

The outline of the paper is the following. We first consider system (\ref{2a}), for which we are able to show the existence of a weak solution using standard technique available for the compressible Navier--Stokes system, see \cite{FNP}.
Next, using a special form of renormalization, and division of equation (\ref{content1}) by $\vr$, we show that we may replace (\ref{content1}) by (\ref{ent_trans}) and  finally by (\ref{content1_1}). We are able to handle (\ref{ent_trans}) as well as (\ref{content1}) for the optimal range of $\gamma$'s (i.e. $\gamma >\frac 32$), while getting equation (\ref{content1_1}) requires the assumption $\gamma\geq\frac 95$. This is a restriction under which the renormalization theory of  DiPerna--Lions \cite{diperna1989ordinary} can be applied.

In Section \ref{s:2}, we introduce the definition of the weak solutions to all three systems mentioned above and present our main existence theorems. Then, in Section \ref{s:3} we recall some specific classical results which are then used in the proof. Further, in Sections \ref{s:4} and \ref{s:5} we prove the existence of weak solutions to system (\ref{2a}); we introduce several levels of  approximations and prove the existence of solutions at each step by performing relevant limit passages in Sections \ref{s:6} and \ref{s:7}. Finally, in Section \ref{s:8} we  prove the existence of weak solution to systems (\ref{1}) and (\ref{1e}).

\section{Weak solutions, existence results}\label{s:2}

Throughout our analysis we naturally distinguish two different situations. They are associated to the magnitude of the adiabatic exponent $\gamma$.
From the point of view of theory of global in time weak solutions, it is reasonable to assume that
\begin{equation}\label{gammacondition}
\gamma > \frac{3}{2}.
\end{equation}
This assumption provides $L^1$ bound of the convective term and is necessary  for application of nowadays techniques.
Under this condition we will first prove the existence of a weak solution to system (\ref{2a}), see Theorem \ref{mainthm2}. Then we shall deduce from this result existence of weak solutions for the formulation (\ref{1e}) still under assumption (\ref{gammacondition}),  see Theorem \ref{mainthm2}.
This result is not equivalent to the existence of weak solutions  to system (\ref{1}) though. The latter can be proved solely   under the restriction
\begin{equation} \label{gammacondition2}
\gamma \geq \frac 95.
\end{equation}
Indeed, the latter more restricted range of $\gamma$'s enables to obtain $L^2$ estimate of the density and, as mentioned in the introduction, makes it possible to apply the DiPerna-Lions theory of the renormalized solutions to the  transport equation (\ref{ent_trans}) and to multiply it by $\vr$ within the class of weak solutions.
\medskip

\subsection{Weak solutions to system (\ref{1})}
Let us first introduce the definition of a weak solution to our original system (\ref{1}). We assume that the initial data (\ref{2}) satisfy:
$$
\vr_0:\Omega\to\R_+,\;\; s_0:\Omega\to \R_+,\;\;\bu_0:\Omega\to\R^3,
$$
\begin{equation}\label{initialdata1}
\vr_0\in L^\gamma(\Omega),\;\;\int_\Omega\vr_0{\rm d}x >0,
\end{equation}
$$
S_0=\vr_0 s_0,\;s_0\in L^\infty(\Omega),\;\;\bm{q}_0=\vr_0\bu_0\in L^{\frac{2\gamma}{\gamma+1}}(\Omega,\R^3). 
$$
The choice of nontrivial  initial condition for $s$ on the set $\{\vr_0=0\}$ will play an important role in the last section. Indeed, there is a certain difference in the proof of the case $s_0 = const$, and $s_0$ non-constant on this set.
We consider

\begin{df}\label{weaksolution}
Suppose the initial conditions satisfy (\ref{initialdata1}).
We say that the triplet  $(\vr,s,\bu)$ is a weak solution of problem (\ref{1})--(\ref{3}) if:
\begin{equation}
(\vr,s,\bu) \in  L^\infty(0,T;L^\gamma(\Omega)) \times L^\infty((0,T)\times \Omega) \times L^2(0,T;W^{1,2}_0(\Omega,\R^3)),
\end{equation}
and for any $t \in [0,T]$ we have:
\begin{description}
\item{(i)}
$ \vr \in C_w([0,T];L^\gamma(\Omega))$ and the continuity equation (\ref{cont1_1}) is satisfied in the weak sense
\begin{equation} \label{weak_contAUX}
\int_\Omega \vr(t,\cdot) \varphi(t,\cdot) \, \dx- \int_\Omega \vr_{0} \varphi(0,\cdot) \, \dx=
\int_0^t \int_\Omega \Big(\vr \partial_t \varphi + \vr \bu \cdot \nabla \varphi\Big)  \, \dx \, \dtau, \forall \varphi \in C^1(\QTb);
\end{equation}
\item{(ii)}
$ \vr s \in C_w([0,T];L^\gamma(\Omega))$ and   equation (\ref{content1_1}) is satisfied in the weak sense
\begin{equation} \label{weak_ZAUX}
\int_\Omega (\vr s)(t,\cdot) \varphi(t,\cdot) \, \dx- \int_\Omega S_{0} \varphi(0,\cdot) \, \dx=
\int_0^t \int_\Omega \Big(\vr s \partial_t \varphi + \vr s \bu \cdot \nabla \varphi\Big)  \, \dx \, \dtau, \forall \varphi \in C^1(\QTb);
\end{equation}
\item{(iii)} $ \vr \bu \in C_w([0,T];L^{\frac{2\gamma}{\gamma+1}}(\Omega,\R^3))$ and
the momentum equation (\ref{mom1_1}) is satisfied in the weak sense
\begin{multline}\label{weak_momAUX}
\int_\Omega (\vr \bu) (t,\cdot) \cdot \bm{\psi}(t,\cdot) \, \dx- \int_\Omega \bm{q}_0 \cdot \bm{\psi}(0,\cdot) \, \dx=
\int_0^t \int_\Omega \Big(\vr \bu \cdot \partial_ t \bm{\psi}  + \vr \bu \otimes \bu : \nabla \bm{ \psi}   \\
+ \vr^\gamma {\cal T}(s) \dv \bm{\psi} - \mathbb{S}(\nabla \bu) : \nabla \bm{ \psi} \Big)  \, \dx \, \dtau, \forall \bm{\psi} \in C_c^1([0,T] \times \Omega,\R^3);
\end{multline}
\item{(iv)}
the energy inequality
\begin{equation}\label{energystep3AUX}
 {\cal E}^1(\vr,s,\bu)(t)  + \int_0^t \int_\Omega \Big(\mu |\nabla \bu|^2 + (\mu +\lambda)(\dv \bu)^2 \Big) \, \dx \, \dtau \le {\cal E}^1(\vr_{0},s_0,\bu_{0})
\end{equation}
holds for a.a $ t \in (0,T) $, where
$$
{\cal E}^1(\vr,s,\bu) = \int_\Omega\Big(\frac 12 \vr |\bu|^2 + \frac{\vr^\gamma {\cal T}(s)}{\gamma-1} \Big)\, \dx.
$$
\end{description}

\end{df}

The first  main result concerning solutions meant by Definition \ref{weaksolution} reads.
\begin{thm}\label{mainthm1}
Let $\mu,\lambda$ satisfy (\ref{viscondition}), $\gamma \geq \frac 95$ and the initial data $(\vr_0,S_0,\bm{q}_0)$ satisfy (\ref{initialdata1}).
Then there exists a weak solution $(\vr,s,\bu)$ to problem (\ref{1})--(\ref{3}) in the sense of Definition \ref{weaksolution}.
\end{thm}

\subsection{Weak solution to system (\ref{2a})}

The restriction on $\gamma$  in Theorem \ref{mainthm1} is obviously not satisfactory as all the physically reasonable values of  $\gamma$ are less or equal that $\frac 53$. We are able to relax this constraint for system (\ref{2a}).
Formally, taking $Z= \vr ({\cal T}(s))^{\frac 1\gamma}$ in (\ref{2a}) one can recover our original system (\ref{1}). However, for the weak solution this formal argument cannot be made rigorous unless we assume that $\gamma\geq \frac 95$.
Nevertheless, system (\ref{2a}) is a good starting point for our considerations.  Indeed, for reasonable initial and boundary conditions it can be shown that it possesses a weak solution for $\gamma>\frac{3}{2}$, using more or less standard approach. Proving existence of solutions directly for system (\ref{1}) seems not to be so simple.

We assume that the initial data for system (\ref{2a}) are
$$
\vr_0:\Omega\to\R_+,\;\; s_0:\Omega\to \R_+,\;\;\bu_0:\Omega\to\R^3,
$$
\begin{equation}\label{2Z}
\vr(0,x)= \vr_0(x),~ Z(0,x)=Z_0(x),~ (\vr\bu)(0,x) = \bm{q}_0(x)=\vr_0\bu_0(x),
\end{equation}
and they satisfy
\begin{equation}\label{initialdata2}
\begin{array}{c}
\displaystyle (\vr_0,Z_0) \in L^\gamma(\Omega)^2, \quad \vr_0, Z_0  \geq 0 \mbox{ a.e. in } \Omega,  \quad \int_\Omega \vr_0 \, \dx>0, \\
\displaystyle 0 \leq c_\star \vr_0 \leq Z_0 \leq c^\star \vr_0 \mbox{ a.e. in } \Omega, \quad 0<c_\star\leq c^\star <\infty,   \quad
\bm{q}_0 \in L^{\frac{2\gamma}{\gamma+1}}(\Omega,\R^3).
\end{array}
\end{equation}

Then we have

\begin{df}\label{weaksolutionaux}
Suppose that the initial conditions satisfy (\ref{initialdata2}).
We say that the triplet  $(\vr,Z,\bu)$ is a weak solution of problem (\ref{2a}) with the initial and boundary conditions (\ref{3}), (\ref{2Z}) if
\begin{equation}
(\vr,Z,\bu) \in  L^\infty(0,T;L^\gamma(\Omega)) \times L^\infty(0,T; L^\gamma(\Omega)) \times L^2(0,T;W^{1,2}_0(\Omega,\R^3)),
\end{equation}
and for any $t \in (0,T]$ we have:
\begin{description}
\item{(i)}
$ \vr \in C_w([0,T];L^\gamma(\Omega))$ and the continuity equation (\ref{cont1}) is satisfied in the weak sense
\begin{equation} \label{weak_cont}
\int_\Omega \vr(t,\cdot) \varphi(t,\cdot) \, \dx- \int_\Omega \vr_{0} \varphi(0,\cdot) \, \dx=
 \int_0^t \int_\Omega \Big(\vr \partial_t \varphi + \vr \bu \cdot \nabla \varphi\Big)  \, \dx \, \dtau, \forall \varphi \in C^1(\QTb);
\end{equation}
\item{(ii)}
$ Z \in C_w([0,T];L^\gamma(\Omega))$ and   equation (\ref{content1}) is satisfied in the weak sense
\begin{equation} \label{weak_Z}
\int_\Omega Z(t,\cdot) \varphi(t,\cdot) \, \dx- \int_\Omega Z_{0} \varphi(0,\cdot) \, \dx=
 \int_0^t \int_\Omega \Big(Z \partial_t \varphi + Z \bu \cdot \nabla \varphi\Big)  \, \dx \, \dtau, \forall \varphi \in C^1(\QTb);
\end{equation}
\item{(iii)} $ \vr \bu \in C_w([0,T];L^{\frac{2\gamma}{\gamma+1}}(\Omega,\R^3))$ and
the momentum equation (\ref{mom1_1}) is satisfied in the weak sense
\begin{multline}\label{weak_mom}
\int_\Omega (\vr \bu) (t,\cdot) \cdot \bm{\psi}(t,\cdot) \, \dx- \int_\Omega \bm{q}_0 \cdot \bm{\psi}(0,\cdot) \, \dx=
\int_0^t \int_\Omega \Big(\vr \bu \cdot \partial_ t \bm{\psi}  + \vr \bu \otimes \bu : \nabla \bm{ \psi}  \\
+ Z^\gamma \dv \bm{\psi} - \mathbb{S}(\nabla \bu) : \nabla \bm{ \psi} \Big)  \, \dx \, \dtau, \forall \bm{\psi} \in C_c^1([0,T] \times \Omega,\R^3);
\end{multline}
\item{(iv)}
the energy inequality
\begin{equation}\label{energystep3}
 {\cal E}^2(\vr,Z,\bu)(t)  + \int_0^t \int_\Omega \Big(\mu |\nabla \bu|^2 + (\mu +\lambda)(\dv \bu)^2\Big)  \, \dx \, \dtau \le {\cal E}^2(\vr_{0},Z_0,\bu_{0})
\end{equation}
holds for a.a $ t \in (0,T) $, where
\begin{equation}\label{E2}
{\cal E}^2(\vr,Z,\bu) = \int_\Omega\Big(\frac 12 \vr |\bu|^2 + \frac{Z^\gamma}{\gamma-1} \Big) \dx.
\end{equation}
\end{description}

\end{df}

Before presenting the existence result for the auxiliary problem, let us recall the definition of  a renormalized solution to equation (\ref{content1}):
\begin{df} \label{renor_cont}
We say that equation  (\ref{content1}) holds in the sense of renormalized  solutions, provided  $(Z,\bu)$, extended by zero outside of $\Omega$, satisfy
\begin{equation}\label{renorent}
\partial_t b(Z) + \dv(b(Z)\bu) + \big(b'(Z)Z-b(Z)\big) \dv \bu = 0~\text{in}~\D'((0,T)\times \R^3),
\end{equation}
where
\begin{equation}\label{regb}
b \in C^1(\R), \quad b'(z)=0,\quad \forall z \in \R~\text{large enough.}
\end{equation}
\end{df}

We have the following existence result for solutions defined by Definition \ref{weaksolutionaux}
\begin{thm}\label{mainthm2}
Let $\mu,\lambda$ satisfy (\ref{viscondition}), $\gamma >\frac 32$, and the initial data $(\vr_0,Z_0,\bm{q}_0)$ satisfy (\ref{initialdata2}).

Then there exists a weak solution $(\vr,Z,\bu)$ to problem (\ref{2a}) with boundary conditions (\ref{3}), in the sense of Definition \ref{weaksolutionaux}. Moreover, $(Z,\bu)$ solves (\ref{content1}) in the renormalized sense and
$$
0 \leq c_\star \vr \leq Z \leq c^\star \vr
$$
a.e. in $(0,T)\times \Omega$.
\end{thm}
\subsection{Weak solution to system (\ref{1e})}
If we replace (\ref{content1_1}) by (\ref{ent_trans}) (using also the renormalization of the latter), the result is also much better than in Theorem \ref{mainthm1}, in fact optimal from the point of view of nowadays theory of compressible Navier--Stokes equations.
In order to formulate the result precisely, we first rewrite system (\ref{1e}) in a slightly different way. We look for a
 triplet $(\rho, \zeta, \bu)$ solving the system of equations
\begin{subequations}\label{wsstep5a}
\begin{equation}\label{contstep5}
\partial_t \vr + \dv(\vr \bu)= 0,
\end{equation}
\begin{equation}\label{transportstep5}
\partial_t \zeta + \bu\cdot \nabla \zeta= 0,
\end{equation}
\begin{equation}\label{momstep5}
\partial_t (\vr \bu) + \dv(\vr \bu \otimes \bu) + \nabla \left(\frac{\vr}{\zeta}\right)^\gamma  = \dv\tn{S}(\nabla \bu),
\end{equation}
\end{subequations}
with initial conditions
\begin{equation}\label{2zeta}
\vr(0,x)= \vr_0(x),~ \zeta(0,x)=\zeta_0(x),~ (\vr\bu)(0,x) = \bm{q}_0(x),
\end{equation}
such that $\zeta_0=\frac{\vr_0}{Z_0}$ and satisfying assumptions (\ref{initialdata2}), in particular
\begin{equation} \label{ICZETA}
 \zeta_0 \in \Big( (c^\star)^{-1},(c_\star)^{-1}\Big).
\end{equation}

Then the weak solution is defined as follows.
\begin{df}\label{weaksolution_without_rho}
Suppose the initial conditions $(\vr_0,\zeta_0, \bm{q}_0)$ satisfy (\ref{ICZETA}) and  (\ref{initialdata2}) (for $\vr_0$ and $\bm{q}_0$),
We say that the triplet  $(\vr,\zeta,\bu)$ is a weak solution of problem (\ref{wsstep5a})   emanating from the initial data $(\vr_0,\zeta_0,\bm{q}_0)$ if
\begin{equation}
(\vr,\zeta ,\bu) \in  L^\infty(0,T;L^\gamma(\Omega)) \times L^\infty((0,T)\times \Omega) \times L^2(0,T;W^{1,2}_0(\Omega,\R^3)),
\end{equation}
and for any $t \in (0,T]$ we have:
\begin{description}
\item{(i)}
$ \vr \in C_w([0,T];L^\gamma(\Omega))$ and the continuity equation (\ref{contstep5}) is satisfied in the weak sense
\begin{equation} \label{weak_contA}
\int_\Omega \vr(t,\cdot) \varphi(t,\cdot) \, \dx- \int_\Omega \vr_{0} \varphi(0,\cdot) \, \dx=
 \int_0^t \int_\Omega \Big(\vr \partial_t \varphi + \vr \bu \cdot \nabla \varphi\Big)  \, \dx \, \dtau, \forall \varphi \in C^1(\QTb);
\end{equation}
\item{(ii)}
$ \zeta \in C_w([0,T];L^\infty(\Omega))$  and   equation (\ref{transportstep5}) is satisfied in the weak sense
\begin{equation} \label{weak_zeta}
\int_\Omega \zeta(T,\cdot) \varphi(T,\cdot) \, \dx- \int_\Omega \zeta_{0} \varphi(0,\cdot) \, \dx=
\int_0^t \int_\Omega \Big(\zeta \partial_t \varphi + \zeta \dv{(\bu  \varphi)}\Big)  \, \dx \, \dtau, \forall \varphi \in C^1(\QTb);
\end{equation}
\item{(iii)} $ \vr \bu \in C_w([0,T];L^{\frac{2\gamma}{\gamma+1}}(\Omega,\R^3))$ and
the momentum equation (\ref{momstep5}) is satisfied in the weak sense
\begin{multline}\label{weak_momA}
\int_\Omega (\vr \bu) (t,\cdot) \cdot \bm{\psi}(t,\cdot) \, \dx- \int_\Omega \bm{q}_0 \cdot \bm{\psi}(0,\cdot) \, \dx=
\int_0^t \int_\Omega \Big(\vr \bu \cdot \partial_ t \bm{\psi}  + \vr \bu \otimes \bu : \nabla \bm{ \psi}   \\
+ \Big(\frac{\vr}{\zeta}\Big)^\gamma \dv \bm{\psi} - \mathbb{S}(\nabla \bu) : \nabla \bm{ \psi} \Big)  \, \dx \, \dtau, \forall \bm{\psi} \in C_c^1([0,T] \times \Omega,\R^3);
\end{multline}
\item{(iv)}
the energy inequality
\begin{equation}\label{energystep3A}
 {\cal E}^2(\vr,\vr/\zeta,\bu)(t)  + \int_0^t \int_\Omega \Big(\mu |\nabla \bu|^2 + (\mu +\lambda)(\dv \bu)^2 \Big) \, \dx \, \dtau \le {\cal E}^2(\vr_{0},\vr_{0}/\zeta_0,\bu_{0})
\end{equation}
holds for a.a $ t \in (0,T) $, where ${\cal E}^2$ is defined through (\ref{E2}).

\end{description}

\end{df}

The last result concerns the existence of solutions meant by Definition \ref{weaksolution_without_rho}.

\begin{thm}\label{mainthm3}
Let $\mu,\lambda$ satisfy (\ref{viscondition}), $\gamma >\frac 32$, and the initial data $(\vr_0,\zeta_0,\bm{q}_0)$ satisfy (\ref{ICZETA}) and (\ref{initialdata2}) (for $\vr_0$ and $\bm{q}_0$).

Then there exists a weak solution $(\vr,\zeta,\bu)$ to problem (\ref{wsstep5a}) with boundary conditions (\ref{3}), in the sense of Definition \ref{weaksolution_without_rho}. Moreover, $(\vr,\bu)$ solves (\ref{contstep5}) and $(\zeta,\bu)$ solves (\ref{transportstep5}) in the renormalized sense.
\end{thm}

Using the result of Theorem \ref{mainthm3}, we may easily obtain a solution to  system (\ref{1e}). Indeed, we may define
$$
s = {\cal T}^{-1}(\zeta^{-\gamma})
$$
and use the fact that equation (\ref{transportstep5}) holds in the renormalized sense.

\begin{rmq}\label{r0}
Note that in two space dimensions, all results hold for any $\gamma >1$. In both two and three space dimensions,  we can also include a non-zero external force on the right-hand side of the momentum equation, i.e. we have additionally the term $\vr\bm{f}$ on the right-hand side of (\ref{mom1_1}), (\ref{mome_1}) and (\ref{mom1}). For $\bm{f} \in L^\infty((0,T)\times \Omega,\R^3)$ we would get the same results as in Theorems \ref{mainthm1}, \ref{mainthm2} and \ref{mainthm3}.
\end{rmq}

\section{Auxiliary results}\label{s:3}

Before proving  our main theorems, we recall several auxiliary results used in this paper. These are mostly standard results and we include them them only for the sake of clarity of presentation.

\begin{lm} \label{KIneq}
Let $\mu>0$, $\lambda + 2\mu >0$. Then there exists a positive constant $c$ such that
\begin{equation}\label{Korn}
\mu \| \nabla \bu \|^2_{L^2(\Omega, \R^{3 \times 3})}  + (\lambda + \mu)\|\dv \bu \|_{L^2(\Omega)} \geq c \|\nabla \bu  \|_{L^2(\Omega, \R^{3 \times 3})}.
\end{equation}
\end{lm}

\begin{lm}\label{strongcvsob}
Let $ \Omega$ be a bounded Lipschitz domain in $\R^3$ . If $g_n \to g $ in $C_w([0,T];L^q(\Omega))$,  $1 < q < \infty$ then $g_n \to g$ strongly in $ L^p(0,T;W^{-1,r}(\Omega)) $ provided $L^q(\Omega) \hookrightarrow \hookrightarrow W^{-1,r}(\Omega)$.
\end{lm}

Note that $L^q(\Omega) \hookrightarrow \hookrightarrow W^{-1,r}(\Omega)$ holds for $\Omega$ a Lipschitz domain in $\R^3$ for $1\leq r\leq \frac 32$ if $q>1$ arbitrary or for $\frac 32 <r<\infty$ provided $q> \frac{3r}{3+r}$.

\begin{lm}\label{boundcontw}
Let $1 \le q< \infty$. Let the sequence $ g_n \in C_{w}([0,T],L^q(\Omega))$ be bounded in $L^\infty(0,T;L^q(\Omega))$. Then it is uniformly bounded on $[0,T]$. More precisely, we have
\begin{equation}
\operatorname{ess \ sup}_{t \in (0,T)} \| g_n(t) \|_{L^q(\Omega)} \le C \Rightarrow \sup_{t \in [0,T]} \| g_n(t) \|_{L^q(\Omega)} \le C,
\end{equation}
where $c$ is a positive constant independent of $n$.
\end{lm}

\begin{lm}\label{compcont}
Let $1<p,q<\infty$ and let $ \Omega$ be a bounded Lipschitz domain in $ \R^3$. Let $ \{g_n\}_{n \in \mathbb{N}}$ be a sequence of functions defined on $ [0,T]$ with values in $L^q(\Omega)$ such that
\begin{equation}\label{l3_a}
g_n \in C_{w}([0,T],L^q(\Omega)),~ g_n~ \text{is uniformly continuous in} ~W^{-1,p}(\Omega)~ \text{and uniformly bounded in}~L^q(\Omega).
\end{equation}
Then, at least for a chosen subsequence
\begin{equation} \label{l3_b}
g_n \to g~ \text{in}~ C_{w}([0,T],L^q(\Omega)).
\end{equation}
If, moreover, $L^q(\Omega) \hookrightarrow \hookrightarrow W^{-1,p}(\Omega)$, then
\begin{equation} \label{l3_c}
g_n \to g \qquad \text{in } C([0,T]; W^{-1,p}(\Omega)).
\end{equation}
\end{lm}

Next, let us consider weak solutions to the continuity equation
\begin{equation} \label{CE3}
\partial_t Z + \dv{(Z\bu)} = 0,\quad Z(0,\cdot) = Z_0(\cdot).
\end{equation}
As a result of the DiPerna--Lions \cite{diperna1989ordinary} theory we have
\begin{lm} \label{ReSolLDP}
Assume $Z \in L^q((0,T)\times\Omega)$ and $\bu \in L^2(0,T;W^{1,2}_0(\Omega))$, where $\Omega \subset \R^3$ is a domain with Lipschitz boundary. Let $(Z,\bu)$ be a weak solution to (\ref{CE3}) and $q\geq 2$. Then $(Z,\bu)$ is also a renormalized solution to (\ref{CE3}), i.e. it solves (\ref{renorent}) in the sense of distributions on $(0,T)\times \R^3$ provided $Z$, $\bu$ are extended by zero outside of $\Omega$.
\end{lm}

\begin{rmq} \label{ReSolExtb}
By density argument and standard approximation technique, we may extend the validity of (\ref{renorent}) to functions $b \in C([0,\infty) \cap C^1(0,\infty))$ such that
\[
|b'(t)| \leq C t^{-\lambda_0}, \qquad \lambda_0 <-1, \quad t \in (0,1],
\]
\[
|b'(t)| \leq C t^{\lambda_1}, \qquad -1<\lambda_1 \leq \frac{q}{2}-1, \quad t \geq 1.
\]
\end{rmq}

\begin{lm} \label{ReSolInitCond}
    Let
	\[
		(s,\mathbf{u})\in \Big(L^{\infty}((0,T)\times \Omega)\cap
		C_{w}([0,T];L^q(\Omega))\Big)\times
		L^2(0,T;W^{1,2}(\Omega,\R^3)) \]
	be a weak solution
	to (\ref{ent_trans}) with $s(0,\cdot)=s_0 \in L^{\infty}(\Omega)$.
	Then for every $B\in C(\R)$, $(B(s),\bu)$ is a distributional solution to
	(\ref{ent_trans}), i.e.
	$$
	\partial_t B(s) + \bu\cdot\nabla B(s) =0
	$$
	in ${\cal D}'((0,T)\times \Omega)$. Moreover, $s$ and $B(s) \in C([0,T]; L^r(\Omega))$ for all $r<\infty$ and
	 $B(s)(0,\cdot)=B(s_0)$.
\end{lm}

In some situations when the DiPerna--Lions theory is not applicable, i.e. when $q<2$ in Lemma \ref{ReSolLDP}, we can still prove that the solution is in fact a renormalized one using the approach from \cite{FeCMUC}. To this purpose  one has to consider the oscillation defect measure of the sequence $Z_\delta$ approximating $Z$, i.e.
\begin{equation} \label{ODF}
{\rm osc}_{q}(Z_\delta-Z)=\sup_{k \in \N} \limsup_{\delta \to 0^+} \|T_k(Z_\delta)-T_k(Z)\|_{L^q((0,T)\times \Omega)},
\end{equation}
where
\begin{equation} \label{TX}
T_k(z) = k T\Big(\frac{z}{k}\Big), \quad z \in \R,\quad k \ge 1,
\end{equation}
with $T \in C^\infty (\R) $ such that
\begin{equation} \label{TkX}
T(z) = z~\text{for}~z\le 1, \quad T(z) =2~\text{for}~z \ge 3,\quad T~\text{concave, non-decreasing}.
\end{equation}
We have
\begin{lm} \label{ReSolFe}
Let $\Omega \subset \R^3$ a domain with Lipschitz boundary. Assume that $(Z_\delta,\bu_\delta)$ is a sequence of renormalized solutions to the continuity equation such that
\[
Z_\delta \to Z \qquad \text{ weakly in } L^1((0,T)\times \Omega),
\]
\[
\bu_\delta \to \bu \qquad \text {weakly in } L^2(0,T; W^{1,2}_0 (\Omega,\R^3))
\]
such that ${\rm osc}_{q}(Z_\delta-Z)<\infty$ for some $q>2$. Then $(Z,\bu)$ is a renormalized solution to the continuity equation.
\end{lm}

We further need the following well-known result \cite{NoSt,DM} concerning the solution operator to the problem
\begin{equation} \label{Bog}
\begin{array}{c}
\dv \bm{v} = f, \\
\bm{v}|_{\partial \Omega}= \bm{0}.
\end{array}
\end{equation}

\begin{lm} \label{Bogovskii}
Let $\Omega$ be a Lipschitz domain in $\R^3$.
For any $1<p<\infty$ there exists a solution operator ${\cal B}\colon \{f \in L^p(\Omega); \int_\Omega f \, \dx= 0\}$ $\to$ $W^{1,p}_0(\Omega,\R^3)$ to (\ref{Bog}) such that for $ \bm{v} = {\cal B}f$ it holds
\[
\|\bm{v}\|_{W^{1,p}(\Omega)} \leq C(\Omega, p)\|f\|_{L^p(\Omega)}.
\]
\end{lm}

Next, we report the following general result concerning the compensated compactness (see \cite{Mu} or \cite{Ta})
\begin{lm} \label{DivCurl}
Let $\bm{U}_n$, $\bm{V}_n$ be two sequences such that
\[
\bm{U}_n \to \bm{U} \quad \text{ weakly in } L^p(\Omega,\R^3),
\]
\[
\bm{V}_n \to \bm{V} \quad \text{ weakly in } L^q(\Omega,\R^3),
\]
where $\frac 1s = \frac 1p + \frac 1q <1$, and
\[
\dv \bm{U}_n \quad \text{is precompact in } W^{-1,r}(\Omega),
\]
\[ \operatorname{curl} {\bm V}_n \quad\text{is precompact in } W^{-1,r}(\Omega,\R^{3\times3})
\]
for a certain $r>0$. Then
$$
\bm{U}_n \cdot \bm{V}_n \to \bm{U} \cdot \bm{V} \qquad \text{ weakly in } L^s(\Omega).
$$
\end{lm}

We will further need the following operators
\begin{equation}
\bm {{\cal A}}[\cdot] = \{\A_i\}_{i=1,2,3}[\cdot]  = \nabla \Delta^{-1}[\cdot],
\end{equation}
where $\Delta^{-1}$ stands for the inverse of the Laplace operator on $\R^3$. To be more specific, the Fourier symbol of $\A_j$ is
\begin{equation}
{\cal F} (\A_j)(\xi) = \frac{-{\rm i}\xi_j}{|\xi|^2}.
\end{equation}
Note that for a sufficiently smooth $v$
\begin{equation}
\sum_{i=1}^3 \partial_{i} \A_i[v] =v
\end{equation}
and, by virtue of the classical Marcinkiewicz multiplier theorem,
\begin{equation}\label{ellipticreg}
\| \nabla \bm {{\cal A}}  [v] \|_{L^{s}(\Omega,\R^3)} \le C(s,\Omega) \| v\|_{L^s(\Omega)}, 1 < s< \infty.
\end{equation}
Note that (see \cite{feireisl2009singular})  if $ v, \partial_t v \in L^p( (0,T) \times \R^3) $, then
\begin{equation}\label{rieszpartial}
\partial_t \bm {{\cal A}}  [v(t,\cdot) ](x) = \bm {{\cal A}} [ \partial_t v(t,\cdot)] (x)~\text{for a.a}~(t,x) \in (0,T) \times \R^3.
\end{equation}
Next, let us also introduce the so-called \textit{Riesz operators}
\begin{equation}
\Rr_{ij}[\cdot] = \partial_{j} \A_i[\cdot] = \partial_j \partial_i \Delta^{-1}[\cdot],
\end{equation}
or, in terms of Fourier symbols, ${\cal F}(\Rr_{ij})(\xi) = \frac{\xi_i \xi_j}{|\xi|^2} $. We recall some of its evident properties needed in the sequel. We have
\begin{equation}\label{R1}
\sum_{i=1}^3 \Rr_{ii}[g] = g, \quad g \in L^r(\R^3), \quad 1 < r < \infty,
\end{equation}
\begin{equation}\label{symetrieriesz}
\int_{\R^3} \Rr_{ij}[u]v \, \dx= \int_{\R^3} u \Rr_{ij}[v] \dx, \quad u \in L^r(\R^3), v \in L^{r'}(\R^3), \quad 1 < r <\infty,
\end{equation}
and
\begin{equation} \label{3.22}
\|\Rr_{ij}[u]\|_{L^p(\R^3)} \leq c(p)\|u\|_{L^p(\R^3)}, \qquad 1<p<\infty.
\end{equation}

\section{Approximation}\label{s:4}

We first focus on the proof of the auxiliary result, i.e. on Theorem \ref{mainthm2}. The problem can be viewed as compressible Navier--Stokes system with two densities, where one is connected with inertia of the fluid and the other one with the pressure. The proof of Theorem  \ref{mainthm2} is hence very similar to the  construction of solutions to the usual  barotropic  Navier--Stokes equations.

The purpose of this section is to introduce subsequent levels of approximation and to formulate relevant existence theorems for each of them. The proofs of these theorems are presented afterwards by performing several limit passages when corresponding  approximation parameters vanish. We first regularize the pressure in order to get higher integrability of $Z$ (and also of $\vr$) in order to obtain the renormalized continuity equations using the DiPerna--Lions technique \cite{diperna1989ordinary}. Next we regularize the continuity equations (for both $\vr$ and $Z$). The construction of a solution is done at another level of approximation, the Galerkin approximation for the velocity.

\subsection{First approximation level}

A weak solution of problem (\ref{2a}) (\ref{3}) is obtained as a limit when $ \delta \to 0^+ $ of the solutions to following problem
\begin{subequations}\label{wsstep2}
\begin{equation}\label{contstep2}
\partial_t \vr + \dv(\vr \bu)= 0,
\end{equation}
\begin{equation}\label{entstep2}
\partial_t Z + \dv(Z \bu)=0,
\end{equation}
\begin{equation}\label{momstep2}
\partial_t (\vr \bu) + \dv(\vr \bu \otimes \bu) + \nabla Z^\gamma + \delta  \nabla Z^\beta = \dv\tn{S}(\nabla \bu)
\end{equation}
\end{subequations}
with the boundary conditions
\begin{equation} \label{uboundary_2}
\bu_{|(0,T)\times \partial \Omega} = \bm{0},
\end{equation}
and modified initial data
\begin{subequations} \label{reg_init}
\begin{equation} \label{reg_init_1}
(\vr(0,\cdot),Z(0,\cdot)) =( \vr_{0,\delta}(\cdot),Z_{0,\delta}(\cdot)) \in C^{\infty}(\overline{\Omega},\R^2), \quad 0 < c_\star \vr_{0,\delta} \le Z_{0,\delta} \le c^\star \vr_{0,\delta} ~\text{in}~ \overline{\Omega},
\end{equation}
\begin{equation} \label{reg_init_2}
\nabla \vr_{0,\delta} \cdot \bn_{|(0,T)\times \partial \Omega} = 0, \quad \nabla Z_{0,\delta} \cdot \bn_{|(0,T)\times \partial \Omega} = 0,
\end{equation}
\begin{equation} \label{reg_init_3}
(\vr \bu) (0,\cdot) = \bm{q}_{0,\delta}(\cdot) \in C^\infty(\overline{\Omega},\R^3).
\end{equation}
\end{subequations}
The specific assumption on the initial data (\ref{reg_init_2}) is not needed here, at this approximation level we would be satisfied with less regular approximation without this condition. However, more regular approximation with the above mentioned compatibility condition is needed at another approximation level and we prefer to regularize the initial condition just once.

Note that we require $\bm{q}_{0,\delta} \to \bm{q}_0$ in $L^{\frac{2\gamma}{\gamma+1}}(\Omega;\R^3)$ and $\vr_{0,\delta} \to \vr_0$, $Z_{0,\delta} \to Z_0$, both in $L^{\gamma}(\Omega)$. While the first part, i.e. the initial condition for the linear momentum, is easy to ensure by standard mollification, the regularization of the initial condition for $Z$ and $\vr$ is more complex. However, we may multiply $Z_0$ by a suitable cut-off function (to set the function to be zero near the boundary), then add a small constant to this function and finally mollify it; i.e.
$$
Z_{0,\delta} = (\varphi_\delta Z_{0}+ \delta)*\omega_\delta.
$$
It is not difficult to see that for suitably chosen cut-off function $\varphi_\delta$\footnote{We may take $\varphi_\delta \in C^\infty_c(\Omega)$ such that $0 \leq \varphi_\delta \leq 1$ in $\Omega$ with $\varphi_\delta(x) = 1$ if (for $x \in \Omega$) $\operatorname{dist}\{x, \partial \Omega\} \geq \frac{\delta}{2}$ and  $\varphi_\delta(x) = 0$ if $\operatorname{dist}\{x, \partial \Omega\} \leq \frac{\delta}{4}$.} all properties connected with $Z_{0,\delta}$  in (\ref{reg_init_1})--(\ref{reg_init_2}) will be fulfilled as well as $Z_{0,\delta} \to Z_0$ in $L^\gamma (\Omega)$ for $\delta \to 0^+$. Similarly we proceed for $\vr_{0}$. By a suitable regularization of the initial linear momentum we may also ensure that
$$
\frac{|\bm{q}_{0,\delta}|^2}{\vr_{0,\delta}} 1_{\{\vr_0 >0\}} \to \frac{|\bm{q}_{0}|^2}{\vr_{0}} 1_{\{\vr_0 >0\}}
$$
in $L^1(\Omega)$.

\subsection{Second approximation level}

We prove the existence of a solution to problem (\ref{wsstep2})--(\ref{reg_init}) by letting $\epsilon \to 0^+$ in the following approximate system.
Given $\epsilon,\delta >0$, we consider
\begin{subequations}\label{wsstep3}
\begin{equation}\label{contstep3}
\partial_t \vr + \dv(\vr \bu)= \epsilon \Delta \vr,
\end{equation}
\begin{equation}\label{entstep3}
\partial_t Z + \dv(Z \bu)= \epsilon \Delta Z,
\end{equation}
\begin{equation}\label{momstep3}
\partial_t (\vr \bu) + \dv(\vr \bu \otimes \bu) + \nabla Z^\gamma + \delta  \nabla Z^\beta + \epsilon \nabla \bu \cdot \nabla \vr = \dv ( \mathbb{S}(\nabla \bu)),
\end{equation}
\end{subequations}
supplemented with the boundary conditions
\begin{equation} \label{uboundary_3_1}
\nabla_x \vr \cdot \bn_{|(0,T)\times \partial \Omega} = 0, \quad \nabla_x Z \cdot \bn_{|(0,T)\times \partial \Omega} = 0,
\end{equation}
\begin{equation} \label{uboundary_3_2}
\bu_{|(0,T)\times \partial \Omega} = \bm{0},
\end{equation}
and modified initial data (\ref{reg_init}) (see the comments above).

\subsection{Existence results for the approximate systems}

Let us present now the existence result for the first approximation level
\begin{prop}\label{step2ex}
Let
$\beta \geq \max(\gamma,4)$, $\delta >0$.
Then, given initial data $(\vr_{0,\delta},Z_{0,\delta},\bu_{0,\delta}) $ as in (\ref{reg_init}), there exists a finite energy weak solution $(\vr,Z,\bu)$ to problem (\ref{wsstep2})--(\ref{reg_init}) such that
\begin{equation}
(\vr,Z,\bu) \in [ L^\infty(0,T;L^\beta(\Omega)) ]^2 \times L^2(0,T;W^{1,2}_0(\Omega,\R^3)),
\end{equation}
\begin{equation}\label{inestep2delta}
0\le c_\star \vr \le Z \le c^\star \vr~\text{a.e in}~ \QT,
\end{equation}
and for any $t \in (0,T)$ we have:
\begin{description}
\item{(i)}
$ \vr \in C_w([0,T];L^\beta(\Omega))$ and the continuity equation (\ref{contstep2}) is satisfied in the weak sense
\begin{equation}\label{contstep2delta}
\int_\Omega \vr(t,\cdot) \varphi(t,\cdot) \, \dx- \int_\Omega \vr_{0,\delta} \varphi(0,\cdot) \, \dx= \int_0^t \int_\Omega \big(\vr \partial_t \varphi + \vr \bu \cdot \nabla \varphi\big) \, \dx\, \dtau, \forall \varphi \in C^1(\QTb);
\end{equation}
\item{(ii)}
$ Z \in C_w([0,T];L^\beta(\Omega))$ and  equation (\ref{entstep2}) is satisfied in the weak sense
\begin{equation}\label{entstep2delta}
\int_\Omega Z(t, \cdot) \varphi(t,\cdot) \, \dx- \int_\Omega Z_{0,\delta} \varphi(0,\cdot) \, \dx= \int_0^t \int_\Omega \big(Z \partial_t \varphi + Z \bu \cdot \nabla \varphi\big) \, \dx\, \dtau, \forall \varphi \in C^1(\QTb);
\end{equation}
\item{(iii)} $ \vr \bu \in C_w([0,T];L^{\frac{2\beta}{\beta+1}}(\Omega,\R^3))$ and
the momentum equation (\ref{momstep2}) is satisfied in the weak sense
\begin{multline}\label{momstep2delta}
\int_\Omega \vr \bu (t, \cdot) \cdot \bm{\psi}(t,\cdot) \, \dx- \int_\Omega \bm{q}_{0,\delta} \cdot \bm{\psi}(0,\cdot) \, \dx= \int_0^t \int_\Omega \Big(\vr \bu \cdot \partial_ t \bm{\psi}  + \vr \bu \otimes \bu : \nabla \bm{ \psi} + Z^\gamma \dv \bm{\psi}   \\
+ \delta Z^\beta \dv \bm{ \psi}- \mathbb{S}(\nabla \bu) : \nabla \bm{ \psi}\Big) \, \dx\, \dtau, \forall \bm{\psi} \in C_c^1([0,T] \times \Omega,\R^3);
\end{multline}
\item{(iv)}
the energy inequality
\begin{equation}\label{energystep2}
 {\cal E}_\delta(\vr,\bu,Z) (t) + \int_0^t \int_\Omega \mathbb{S}(\nabla \bu) : \nabla \bu  \, \dx\, \dtau \le \E_\delta(\vr_{0,\delta},Z_{0,\delta},\bu_{0,\delta})
\end{equation}
holds for a.a $ t \in (0,T) $, where ${\cal E}_\delta(\vr,\bu,Z) = \int_\Omega \big(\frac 12 \vr |\bu|^2 + \frac{\delta}{\beta-1} Z^\beta+ \frac{1}{\gamma-1} Z^\gamma\big)\, \dx$;
\item{(v)}
the following estimates hold with constants independent of $\delta$
\begin{equation}\label{eststep21}
\sup_{t \in [0,T]} \| \vr(t) \|^\gamma_{L^\gamma(\Omega)}+ \sup_{t \in [0,T]} \| Z(t) \|^\gamma_{L^\gamma(\Omega)} \le C(\gamma,c_\star) \E_\delta (\vr_{0,\delta},Z_{0,\delta},\bu_{0,\delta}),
\end{equation}
\begin{equation}\label{eststep22}
\delta \sup_{t \in [0,T]} \| \vr(t) \|_{L^\beta(\Omega)}^\beta + \delta \sup_{t \in [0,T]} \| Z(t) \|_{L^\beta(\Omega)}^\beta \le C(\beta,c_\star) \E_\delta (\vr_{0,\delta},Z_{0,\delta},\bu_{0,\delta}),
\end{equation}
\begin{equation}\label{eststep23}
\| \bu \|_{L^2(0,T;W^{1,2}_0(\Omega,\R^3))} \le C \E_\delta (\vr_{0,\delta},Z_{0,\delta},\bu_{0,\delta}),
\end{equation}
\begin{equation}\label{eststep24}
\sup_{t\in [0,T]} \| \vr \bu \|_{L^{\frac{2\gamma}{\gamma+1}}(\Omega,\R^3))}+ \sup_{t\in [0,T]} \| Z \bu \|_{L^{\frac{2\gamma}{\gamma+1}}(\Omega,\R^3)} \le  C(\gamma,c^\star, \E_\delta (\vr_{0,\delta},Z_{0,\delta},\bu_{0,\delta})),
\end{equation}
\begin{equation}\label{eststep25}
\| \vr \bu  \|_{L^2(0,T;L^{\frac{6\gamma}{\gamma+6}}(\Omega,\R^3))}+ \| Z \bu \|_{L^2(0,T;L^{\frac{6\gamma}{\gamma+6}}(\Omega,\R^3))} \le C(\gamma,c_\star,\E_\delta (\vr_{0,\delta},Z_{0,\delta},\bu_{0,\delta})),
\end{equation}
\begin{equation}\label{eststep26}
\| \vr |\bu |^2 \|_{L^1(0,T;L^{\frac{3\gamma}{\gamma+3}}(\Omega))}+ \| Z |\bu|^2 \|_{L^1(0,T;L^{\frac{3\gamma}{\gamma+3}}(\Omega))} \le  C(\gamma,c_\star,\E_\delta (\vr_{0,\delta},Z_{0,\delta},\bu_{0,\delta})),
\end{equation}
\begin{equation}\label{eststep27}
\| \vr |\bu |^2 \|_{L^2(0,T;L^{\frac{6\gamma}{4\gamma+3}}(\Omega))}+ \| Z |\bu|^2 \|_{L^2(0,T;L^{\frac{6\gamma}{4\gamma+3}}(\Omega))} \le  C(\gamma,c_\star,\E_\delta (\vr_{0,\delta},Z_{0,\delta},\bu_{0,\delta})),
\end{equation}
\begin{equation} \label{eststep28}
\begin{array}{c}
\vspace{0.2cm}
\|\vr\|^{\gamma+\theta}_{L^{\gamma+\theta}(\QT)} + \delta \|\vr\|^{\beta+\theta}_{L^{\beta+\theta}(\QT)} + \|Z\|^{\gamma+\theta}_{L^{\gamma+\theta}(\QT)} \\
+ \delta \|Z\|^{\beta+\theta}_{L^{\beta+\theta}(\QT)} \leq  C(\gamma,c_\star,\E_\delta (\vr_{0,\delta},Z_{0,\delta},\bu_{0,\delta})),
\end{array}
\end{equation}
\end{description}
where $\theta = \min\{\frac 23 \gamma-1,\frac{\gamma}{2}\}$. Moreover, equations  (\ref{contstep2}), (\ref{entstep2}) hold in the sense of renormalized solutions in $ \D'((0,T)\times \Omega)$ and $ \D'((0,T)\times \R^3)$  provided $\vr,Z,\bu $  are prolonged by zero outside $\Omega$.
\end{prop}

We have for the second approximation level

\begin{prop}\label{step3ex}
Suppose $ \beta \geq \max(4,\gamma)$. Let  $\epsilon$, $\delta >0$.  Assume the initial data $ (\vr_{0,\delta},Z_{0,\delta},\bu_{0,\delta})$ satisfy (\ref{reg_init}). Then there exists a weak solution $ (\vr,Z,\bu)$ to  problem (\ref{reg_init})--(\ref{uboundary_3_2}) such that
\begin{equation}
(\vr,Z,\bu) \in [ L^\infty(0,T;L^\beta(\Omega)) \cap L^2(0,T; W^{1,2}(\Omega)) ]^2 \times L^2(0,T;W_0^{1,2}(\Omega,\R^3)),
\end{equation}
\begin{equation}\label{inestep3}
0\le c_\star \vr \le Z \le c^\star \vr~\text{a.e in}~ \QT,
\end{equation}
and for any $t \in (0,T)$ we have:
\begin{description}
\item{(i)}
$ \vr \in C_w([0,T];L^\beta(\Omega))$ and the continuity equation (\ref{contstep3}) is satisfied in the weak sense
\begin{equation}\label{contstepep}
\begin{array}{c}
\displaystyle
\int_\Omega \vr(t, \cdot) \varphi(t,\cdot) \, \dx- \int_\Omega \vr_{0,\delta} \varphi(0,\cdot) \, \dx\\
\displaystyle = \int_0^t \int_\Omega \Big(\vr \partial_t \varphi + \vr \bu \cdot \nabla \varphi - \epsilon \nabla \vr \cdot \nabla \varphi\Big) \, \dx \, \dtau,\quad \forall \varphi \in C^1(\QTb);
\end{array}
\end{equation}
\item{(ii)}
$ Z \in C_w([0,T];L^\beta(\Omega))$ and equation (\ref{entstep3}) is satisfied in the weak sense
\begin{equation}\label{entstepep}
\begin{array}{c}
\displaystyle \int_\Omega Z(t, \cdot) \varphi(t,\cdot) \, \dx- \int_\Omega Z_{0,\delta} \varphi(0,\cdot) \, \dx\\
\displaystyle = \int_0^t \int_\Omega \Big(Z \partial_t \varphi + Z \bu \cdot \nabla \varphi -\epsilon \nabla Z \cdot \nabla \varphi\Big)  \, \dx \, \dtau, \quad \forall \varphi \in C^1(\QTb);
\end{array}
\end{equation}
\item{(iii)} $ \vr \bu \in C_w([0,T];L^{\frac{2\beta}{\beta+1}}(\Omega,\R^3))$ and
the momentum equation ($\ref{momstep3}$) is satisfied in the weak sense
\begin{multline}\label{momstepep}
\int_\Omega \vr \bu (t, \cdot) \cdot \bm{\psi}(t,\cdot) \, \dx- \int_\Omega \bm{q}_{0,\delta} \cdot \bm{\psi}(0,\cdot) \, \dx= \int_0^t \int_\Omega \Big(\vr \bu \cdot \partial_ t \bm{\psi}  + \vr \bu \otimes \bu : \nabla \bm{ \psi} + Z^\gamma \dv \bm{\psi}  \\
+ \delta Z^\beta \dv \bm{ \psi} - \mathbb{S}(\nabla \bu) : \nabla \bm{ \psi} + \epsilon \nabla \vr \cdot \nabla \bu \cdot \bm{\psi}\Big)\, \dx \, \dtau,\quad \forall \bm{\psi} \in C_c^1([0,T] \times \Omega,\R^3);
\end{multline}
\item{(iv)}
the energy inequality
\begin{multline}\label{energystep3ep}
{\cal E}_\delta(\vr, \bu, Z)(t) + \int_0^t \int_\Omega \Big(\mathbb{S}(\nabla \bu) : \nabla \bu + \frac{\epsilon \gamma}{\gamma-1} Z^{\gamma-2} |\nabla Z|^2  \\  + \frac{\epsilon\delta \beta}{\beta-1} Z^{\beta-2} |\nabla Z|^2\Big) \, \dx \, \dtau \le \E_\delta(\vr_{0,\delta},Z_{0,\delta},\bu_{0,\delta})
\end{multline}
holds for a.a $ t \in (0,T) $, where ${\cal E}_\delta(\vr, \bu, Z) $ is the same as in Proposition \ref{step2ex};
\item{(v)}
the following estimates hold with constants independent of $\epsilon$
\begin{equation}\label{eststep32}
\sup_{t \in [0,T]} \| \vr(t) \|_{L^\beta(\Omega)}^\beta +  \sup_{t \in [0,T]} \| Z(t) \|_{L^\beta(\Omega)}^\beta \le C(\beta,c_\star) \E_\delta (\vr_{0,\delta},Z_{0,\delta},\bu_{0,\delta}),
\end{equation}
\begin{equation}\label{eststep33}
  \|  \bu \|_{L^2(0, T; W^{1,2}_0(\Omega,\R^3))}  \le C\E_\delta (\vr_{0,\delta},Z_{0,\delta},\bu_{0,\delta}),
\end{equation}
\begin{equation}\label{eststep36}
\sup_{t\in [0,T]} \| \vr \bu \|_{L^{\frac{2\beta}{\beta+1}}(\Omega,\R^3))}+ \sup_{t\in [0,T]} \| Z \bu \|_{L^{\frac{2\beta}{\beta+1}}(\Omega,\R^3)}  \le  C(\beta,c_\star,\E_\delta (\vr_{0,\delta},Z_{0,\delta},\bu_{0,\delta})),
\end{equation}
\begin{equation}\label{eststep34}
\epsilon  \Big( \| \nabla \vr\|_{L^2((0,T)\times\Omega,\R^3)}^2 + \| \nabla Z \|_{L^2((0,T)\times\Omega,\R^3)}^2\Big) \le C (\beta,c_\star, {\cal E}_\delta(\vr_{0,\delta },Z_{0,\delta},\bu_{0,\delta})),
\end{equation}
\begin{equation}\label{eststep34a}
\| \vr |\bu |^2 \|_{L^2(0,T;L^{\frac{6\beta}{4\beta+3}}(\Omega))}+ \| Z |\bu|^2 \|_{L^2(0,T;L^{\frac{6\beta}{4\beta+3}}(\Omega))} \le  C(\beta,c_\star,\E_\delta (\vr_{0,\delta},Z_{0,\delta},\bu_{0,\delta})),
\end{equation}
\begin{equation} \label{eststep35}
\| \vr \|_{L^{\beta+1}(\QT)} + \| Z \|_{L^{\beta+1}(\QT)} \le C (\beta,c_\star, {\cal E}_\delta(\vr_{0,\delta },Z_{0,\delta},\bu_{0,\delta})).
\end{equation}
\end{description}
\end{prop}

\section{Existence for the second approximation level}\label{s:5}

We are not going to present detailed proof of Proposition \ref{step3ex}, as it is similar to the corresponding step in the existence proof for the barotropic Navier--Stokes equations, cf. \cite{NoSt}. In what follows we only explain main ideas as well as how to obtain the crucial estimate (\ref{inestep3}).

We introduce another approximation level, the  Galerkin approximation for the velocity. We take a suitable basis $\{\bm{\Phi}_j\}_{j=1}^\infty$ in $W^{1,2}_0(\Omega,\R^3)$, orthonormal in $L^2(\Omega, \R^3)$, and replace (\ref{momstepep}) by
\begin{multline}\label{momstepgal}
 \int_\Omega \partial_t(\vr \bu^n) \cdot  \bm{\Phi}_j \, \dx=   \int_\Omega \Big(\vr \bu^n \otimes \bu^n : \nabla \bm{ \Phi}_j + Z^\gamma \dv \bm{\Phi}_j
+ \delta Z^\beta \dv \bm{ \Phi}_j \\ - \mathbb{S}(\nabla \bu^n) : \nabla \bm{ \Phi}_j + \epsilon \nabla \vr \cdot \nabla \bu^n \cdot \bm{\Phi}_j\Big)\, \dx, \quad\forall j=1,2,\dots, n,
\end{multline}
where $\vr$ and $Z$ solves (\ref{contstep3}) and (\ref{entstep3}), respectively, with $\bu$ replaced by $\bu^n$, and
$$
\bu^n(t,x) = \sum_{j=1}^n a^n_j(t) \bm{\Phi}_j(x).
$$
The initial condition for the momentum equation reads
$$
\vr(0,\cdot) \bu^n(0,\cdot) = P^n(\bm{q}_{0,\delta})(\cdot)
$$
with $P^n$ the corresponding orthogonal projection on the space spanned by $\{\bm{\Phi}_j\}_{j=1}^n$. We construct the solutions to the $n$-th Galerkin approximation by means of a version of the Schauder fixed point theorem. The fundamental step in this procedure is derivation of the a priori estimates. They can be obtained by using the solution $\bu^n$ as a test function in (\ref{momstepgal}) and combining it with (\ref{entstep3}) as well as  with (\ref{contstep3}). We then deduce
\begin{multline} \label{energy_gal}
{\cal E}_{\delta}(\vr,Z,\bu^n)(t) + \int_0^t\int_{\Omega} \Big({\tn S}(\nabla \bu^n):\nabla \bu^n + \epsilon\gamma Z^{\gamma-2}|\nabla Z|^2 + \epsilon\delta  \beta Z^{\beta-2} |\nabla Z|^2\Big) \, \dx \, \dtau  \\ \leq  {\cal E}_{\delta}(\vr_{0,\delta},Z_{0,\delta},P^n(\bm{q}_{0,\delta})/\vr_{0,\delta} ) \leq C
\end{multline}
with $C$ independent of $n$ (also of $\epsilon$ and $\delta$). Next, testing equations (\ref{contstep3}) and (\ref{entstep3}) by $\vr$ and $Z$, respectively, we also have
\begin{equation} \label{estimate1_gal}
\|\vr\|_{L^2(\Omega)}^2(t) + \|Z\|_{L^2(\Omega)}^2 (t) + \epsilon \int_0^t \Big(\|\nabla \vr\|_{L^2(\Omega;\R^3)}^2 + \|\nabla Z\|_{L^2(\Omega;\R^3)}^2 \Big)\dtau \leq C
\end{equation}
provided $\beta \geq 4$. Note also that
$$
\frac{{\rm d}}{\dt}\int_{\Omega} \vr \, \dx= \frac{{\rm d}}{\dt}\int_{\Omega} Z \, \dx= 0.
$$
To prove inequalities (\ref{inestep3})  we use a simple comparison principle between $\vr$ and $Z$. Taking $c_\star,c^\star$ as in (\ref{reg_init_1}) we may write
$$
\partial_t(Z-c_\star \vr ) + \dv{\big(\bu^n(Z-c_\star \vr )\big)} - \epsilon \Delta (Z-c_\star \vr ) = 0
$$
and
$$
\partial_t(c^\star \vr-Z) + \dv\big(\bu^n (c^\star \vr-Z)\big) - \epsilon \Delta (c^\star \vr-Z) = 0.
$$
As both equations have non-negative initial conditions, it is easy to see that also the solutions are non-negative and due to the uniqueness of solutions we deduce that
\begin{equation} \label{max_princ_gal}
0< c_\star \vr \leq Z \leq c^\star \vr < \infty
\end{equation}
a.e. in $\QT$. Combining (\ref{max_princ_gal}) with (\ref{energy_gal}) we also have
\begin{equation} \label{estimate2_gal}
\|\vr\|_{L^\infty(0,T; L^\beta(\Omega))} \leq C
\end{equation}
with $C= C(c_\star,\delta,\E_\delta)$. The regularity of solutions to parabolic problems allows us to deduce that we have independently of $n$
\begin{equation} \label{estimate3_gal}
\|\partial_t \vr\|_{L^q(0,T;L^q(\Omega))} + \|\partial_t Z\|_{L^q(0,T;L^q(\Omega))} + \|\vr\|_{L^q(0,T;W^{2,q}(\Omega))} + \|Z\|_{L^q(0,T;W^{2,q}(\Omega))} \leq C(\epsilon)
\end{equation}
for all $q \in (1,\infty)$.
These estimates are sufficient to apply the fixed point argument, but also to pass to the limit $n \to \infty$. To this aim, recall also that $\vr$ and $Z$ belong to $C_{w}([0,T]; L^\beta(\Omega))$ and $\vr \bu^n$ to $C_{w}([0,T]; L^{\frac{2\beta}{\beta+1}}(\Omega,\R^3))$. Hence, using several general results from Section 3 (see Lemmas \ref{strongcvsob}--\ref{compcont}) we may pass to the limit with $n \to \infty$ to recover system (\ref{reg_init})--(\ref{uboundary_3_2}) as stated in Proposition \ref{step3ex}. To finish the proof of this proposition, we have to show estimate (\ref{eststep35}).  To this aim, we use as test function in (\ref{momstepep}) $\bm{\psi}$, solution to (cf. Lemma \ref{Bogovskii} in Section 3)
$$
\dv \bm{\psi} = Z - \frac{1}{|\Omega|} \int_{\Omega} Z\,  \dx
$$
with homogeneous Dirichlet boundary conditions. Due to properties of the Bogovskii operator we may prove
$$
\|Z\|_{L^{\beta+1}(\QT)} \leq C
$$
which, together with (\ref{max_princ_gal}), finishes the proof of Proposition \ref{step3ex}.

\section{Vanishing viscosity limit: proof of Proposition \ref{step2ex}}\label{s:6}

\subsection{Limit passage based on the a priori estimates}

At this stage, we are ready to pass to the limit for $\epsilon \to 0^+$ to get rid of the diffusion term in the equations (\ref{contstep3}), (\ref{entstep3}) as well as of the $\epsilon$-dependent term in (\ref{momstep3}). Note that the parameter $\delta $ is kept fixed throughout this procedure so that we may use the estimates derived above, except (\ref{estimate3_gal}). Accordingly, the solution of problem (\ref{reg_init})--(\ref{uboundary_3_2}) obtained in Proposition \ref{step3ex} above will be denoted $(\vr_\epsilon,Z_\epsilon,\bu_\epsilon)$.

First of all, by virtue of (\ref{eststep33}) and (\ref{eststep34}), we obtain
$$
\epsilon \nabla \vr_\epsilon \cdot \nabla \bu_\epsilon \to 0~\text{in}~L^1(\QT),
$$
and, analogously,
$$
\epsilon \nabla Z_\epsilon, \ \epsilon \nabla \vr_\epsilon \to 0~\text{in}~L^2(\QT).
$$

From estimates (\ref{eststep32})--(\ref{eststep35}) we further deduce
\begin{subequations}
\begin{equation}\label{contconvstep3}
\vr_\epsilon \to \vr~\text{weakly-}\star~\text{in}~L^\infty(0,T;L^\beta(\Omega))~\text{and weakly in}~ L^{\beta+1}(\QT),
\end{equation}
\begin{equation}\label{contconvstep3ent}
Z_\epsilon \to Z~\text{weakly-}\star~\text{in}~L^\infty(0,T;L^\beta(\Omega))~\text{and weakly in}~ L^{\beta+1}(\QT),
\end{equation}
\begin{equation}\label{weakconvvelstep3}
\bu_\epsilon \to \bu ~\text{ weakly in}~L^2(0,T;W^{1,2}_0(\Omega,\R^3)),
\end{equation}
\end{subequations}
passing to subsequences if necessary.

By virtue of (\ref{inestep3}) and the weak $L^{\beta+1}$-convergence derived above we obtain
\begin{equation}
0 \le c_\star \vr \le Z \le c^\star \vr~\text{a.e. in}~ \QT.
\end{equation}
Due to (\ref{contstepep}), (\ref{entstepep}), (\ref{eststep34}) and (\ref{eststep35}), $\vr_\epsilon $ and $Z_\epsilon$ are uniformly continuous in $ W^{-1,\frac{2\beta}{\beta+1}}(\Omega)$. Since they belong to $C_w([0,T];L^\beta (\Omega)) $  and they are uniformly bounded in $L^\beta (\Omega)$ (by virtue of (\ref{eststep32})), we use Lemma \ref{compcont}, in order to get at least for a chosen subsequence
\begin{equation}
\vr_\epsilon \to \vr, \quad Z_\epsilon \to Z~\text{in}~C_w([0,T];L^\beta (\Omega)).
\end{equation}

Once we realize that the imbedding $ L^s (\Omega) \hookrightarrow W^{-1,2}(\Omega) $ is compact for $s >\frac{6}{5}$, we apply Lemma \ref{strongcvsob} to $\vr_\epsilon $ and $Z_\epsilon$, and  obtain
\begin{equation}
\vr_\epsilon \to \vr,\quad Z_\epsilon \to Z~\text{in}~L^p (0,T;W^{-1,2}(\Omega)), \quad 1 \le p < \infty.
\end{equation}
Consequently, by virtue of the previous formula, (\ref{eststep36}) and (\ref{weakconvvelstep3}) we obtain
\begin{equation}
\vr_\epsilon \bu_\epsilon \to \vr \bu,\quad Z_\epsilon \bu_\epsilon \to Z \bu~\text{weakly-$\star$ in}~L^\infty (0,T; L^{\frac{2\beta}{\beta+1}}(\Omega,\R^3)).
\end{equation}

Taking into account (\ref{momstepep}) and (\ref{eststep32})--(\ref{eststep35}) we conclude that $ \vr_\epsilon \bu_\epsilon $ is uniformly continuous in $W^{-1,s}(\Omega, \R^3)$, where $ s = \frac{\beta+1}{\beta} $. Since it belongs to $ C_w([0,T];L^{\frac{2\beta}{\beta+1}}(\Omega,\R^3))$ and since it is uniformly bounded in $ L^{\frac{2\beta}{\beta+1}}(\Omega,\R^3)$ (see (\ref{eststep36})), Lemma \ref{compcont} yields
\begin{equation}
\vr_\epsilon \bu_\epsilon \to \vr \bu~\text{in}~C_w([0,T];L^{\frac{2\beta}{\beta+1}}(\Omega,\R^3)).
\end{equation}
The imbedding $ L^{\frac{2\beta}{\beta+1}}(\Omega) \hookrightarrow W^{-1,2}(\Omega) $ is compact, hence we deduce from Lemma \ref{strongcvsob}
\begin{equation}
\vr_\epsilon \bu_\epsilon \to \vr \bu~\text{strongly in}~L^p(0,T;W^{-1,2}(\Omega,\R^3)).
\end{equation}
It implies, together with  (\ref{weakconvvelstep3}) that
\begin{equation}
\vr_\epsilon \bu_\epsilon \otimes \bu_\epsilon \to \vr \bu \otimes \bu~\text{in}~ L^q((0,T)\times \Omega;\R^{3\times 3})
\end{equation}
for some $q>1$.

We have proven that the limits $\vr$, $Z$ and $\bu$ satisfy for any $t\in[0,T]$ the following system of equations
\begin{equation}
\int_\Omega \vr(t, \cdot) \varphi(t,\cdot) \, \dx- \int_\Omega \vr_{0,\delta} \varphi(0,\cdot)\, \dx=
 \int_0^t \int_\Omega \Big(\vr \partial_t \varphi + \vr \bu \cdot \nabla \varphi\Big) \, \dx  \, \dtau, \forall \varphi \in C^1 (\QTb);
\end{equation}
\begin{equation}\label{entstep3limit}
\int_\Omega Z(t, \cdot) \varphi(t,\cdot) \, \dx- \int_\Omega Z_{0,\delta} \varphi(0,\cdot)\, \dx=
\int_0^t \int_\Omega \Big(Z \partial_t \varphi + Z \bu \cdot \nabla \varphi\Big) \, \dx \, \dtau, \forall \varphi \in C^1 (\QTb);
\end{equation}
\begin{multline}
\int_\Omega \vr \bu(t. \cdot) \cdot \bm{\psi}(t,\cdot) \, \dxdt - \int_\Omega \bm{q}_{0,\delta}  \cdot \bm{\psi}(0,\cdot) \, \dxdt
= \int_0^t \int_\Omega \Big(\vr \bu \cdot \partial_ t \bm{\psi}  + \vr \bu \otimes \bu : \nabla \bm{ \psi} + \Ov{p} \dv \bm{ \psi}  \\
\\ - \mathbb{S}(\nabla \bu) : \nabla \bm{ \psi} \Big)\, \dx  \, \dtau, \forall \bm{\psi} \in C_c^1([0,T] \times \Omega,\R^3),
\end{multline}
where, by virtue of  (\ref{eststep35}),
\begin{equation}\label{weakcvpressurestep3}
Z_\epsilon^\gamma+ \delta Z_\epsilon^\beta \to \Ov{p}~\text{weakly in }~L^{\frac{\beta+1}{\beta}}(\QT).
\end{equation}
In particular, equations (\ref{contstep3}), (\ref{entstep3}) and (\ref{momstep3}) (with $\Ov{p}$ instead of $Z^\gamma+ \delta Z^\beta$) are satisfied in the sense of distributions and  the limit functions satisfy the initial condition
\begin{equation}
\vr(0,\cdot)= \vr_{0,\delta}(\cdot), \quad Z(0,\cdot)=Z_{0,\delta}(\cdot),\quad (\vr \bu) (0,\cdot)=\bm{q}_{0,\delta}(\cdot),
\end{equation}
where $(\vr_{0,\delta}, Z_{0,\delta}, \bm{q}_{0,\delta}) $ are defined in (\ref{reg_init}).

Thus our ultimate goal is to show that
\begin{equation}\label{pressurerelation}
\Ov{p}= Z^\gamma + \delta Z^\beta
\end{equation}
which is equivalent to the strong convergence of $Z_\epsilon$ in $L^1(\QT)$.

\subsection{Effective viscous flux}

We introduce the quantity $Z^\gamma+\delta Z^\beta -(\lambda+2\mu)\dv \bu $ called usually the effective viscous flux. This quantity enjoys remarkable properties for which we refer to Hoff \cite{hoff1995strong}, Lions \cite{lions1998mathematical}, or Serre \cite{serre1991variations}. We have the following crucial result.
\begin{lm}\label{effectivestep3}
Let $\vr_\epsilon,Z_\epsilon,\bu_\epsilon$ be the sequence of approximate solutions, the existence of which is guaranteed by Proposition \ref{step3ex}, and let $\vr,Z,\bu$ and $\Ov{p}$ be the limits appearing in (\ref{contconvstep3}), (\ref{contconvstep3ent}), (\ref{weakconvvelstep3})  and (\ref{weakcvpressurestep3}) respectively.
Then
\begin{equation*}
\lim_{\epsilon \to 0^+} \int_0^T \psi \int_\Omega \phi  \Big( Z_\epsilon^\gamma+ \delta Z_\epsilon^\beta-(\lambda+2\mu)\dv \bu_\epsilon \Big) Z_\epsilon \, \dxdt \\
= \int_0^T \psi \int_\Omega \phi  \Big( \Ov{p}-(\lambda+2\mu)\dv \bu \Big) Z \, \dxdt
\end{equation*}
for any $\psi \in C_c^\infty ((0,T))$ and $\phi \in C_c^\infty(\Omega)$, passing to subsequences, if necessary.
\end{lm}

The proof of Lemma \ref{effectivestep3} is based on the Div-Curl Lemma of compensated compactness, see Lemma \ref{DivCurl}. We will not present it here, as it is a relatively standard result in the theory of weak solutions to the compressible Navier-Stokes equations; see e.g. \cite{NoSt} for more details. The basic tools for the proof can be found in Section 3.  We shall give more details to the proof of a similar result used in the limit passage $\delta \to 0$, where, moreover, several arguments are more subtle than here.

We conclude this section by showing (\ref{pressurerelation}) and, consequently, strong convergence of the sequence $Z_\epsilon$ in $L^1(\QT)$.

Recall that $Z$ solves (\ref{entstep3limit})  in the sense of renormalized equations, see Lemma \ref{ReSolLDP}. Thus, we take $b(Z) = Z \ln Z$ (see Remark \ref{ReSolExtb}) to get

\begin{equation}\label{strongconvstep1}
\int_0^T \int_\Omega Z \dv_x \bu \, \dxdt = \int_\Omega  Z_{0,\delta} \ln(Z_{0,\delta}) \, \dx- \int_\Omega Z(T) \ln \big(Z(T)\big)\, \dx.
\end{equation}
On the other hand, $Z_\epsilon$ solves (\ref{entstep3}) a.e on $\QT$, in particular,
\begin{equation*}
\partial_t b(Z_\epsilon) + \dv_x( b(Z_\epsilon) \bu_\epsilon) + \big( b'(Z_\epsilon)Z_\epsilon-b(Z_\epsilon)\big) \dv \bu_\epsilon - \epsilon \Delta b(Z_\epsilon) \le 0
\end{equation*}
for any $b$ convex and globally Lipschitz on $\R^+$; whence
\begin{equation*}
\int_0^T \int_\Omega \Big(b'(Z_\epsilon) Z_\epsilon - b(Z_\epsilon)\Big) \dv \bu_\epsilon \, \dxdt \le \int_\Omega b(Z_{0,\delta}) \, \dx- \int_\Omega b\big(Z_\epsilon(T)\big) \, \dx
\end{equation*}
from which we easily deduce
\begin{equation}\label{strongconvstep2}
\int_0^T \int_\Omega Z_\epsilon \dv \bu_\epsilon \, \dxdt \le \int_\Omega Z_{0,\delta} \ln(Z_{0,\delta}) \, \dx- \int_\Omega Z_\epsilon(T) \ln\big(Z_\epsilon(T)\big)\, \dx.
\end{equation}
Note that
$$
\int_{\Omega} Z(T)\ln (Z(T)) \, \dx\leq \liminf_{\epsilon \to 0^+} \int_{\Omega} Z_\epsilon (T) \ln \big(Z_\epsilon(T)\big)\, \dx.
$$
Take two non-decreasing sequences $ \psi_n, \phi_n $ of non-negative functions such that
\begin{equation}
\psi_n \in C_c^\infty(0,T), \psi_n \to 1, \phi_n \in C_c^\infty(\Omega), \phi_n \to 1.
\end{equation}
Lemma  \ref{effectivestep3} implies that
\begin{multline*}
\limsup_{\epsilon \to 0^+} \int_0^T \psi_m \int_\Omega \phi_m (Z_\epsilon^\gamma +\delta Z_\epsilon^\beta) Z_\epsilon \, \dxdt \le \limsup_{\epsilon \to 0^+} \int_0^T \psi_n \int_\Omega \phi_n (Z_\epsilon^\gamma +\delta Z_\epsilon^\beta) Z_\epsilon \, \dxdt \\
\le \lim_{\epsilon \to 0^+} \int_0^T \psi_n \int_\Omega \phi_n \big( Z_\epsilon^\gamma + \delta Z_\epsilon^\beta -(\lambda+2\mu) \dv \bu_\epsilon \big) Z_\epsilon \, \dxdt \\
+ (\lambda+2\mu) \limsup_{\epsilon \to 0^+}  \int_0^T \psi_n \int_\Omega \phi_n  Z_\epsilon \dv \bu_\epsilon \, \dxdt
\le \int_0^T \psi_n \int_\Omega \phi_n ( \Ov{p} -(\lambda+2\mu) \dv_x \bu ) Z \, \dxdt \\
+(\lambda+2\mu) \limsup_{\epsilon \to 0^+} \int_0^T \int_\Omega Z_\epsilon | 1 -\psi_n \phi_n | | \dv \bu_\epsilon | \, \dxdt  +(\lambda+2\mu) \limsup_{\epsilon \to 0^+} \int_0^T \int_\Omega Z_\epsilon \dv \bu_\epsilon \, \dxdt.
\end{multline*}
Using also ($\ref{strongconvstep1}$) and ($\ref{strongconvstep2}$), we observe that
\begin{multline*}
 \limsup_{\epsilon \to 0^+} \int_0^T \psi_m \int_\Omega \phi_m (Z_\epsilon^\gamma +\delta Z_\epsilon^\beta) Z_\epsilon \, \dxdt \\
\le
\int_0^T \int_\Omega \Ov{p} Z \, \dxdt
+ \eta(n) +(\lambda + 2\mu) \Big[ \int_\Omega Z (T) \ln (Z(T)) \, \dx- \limsup_{\epsilon \to 0^+} \int_\Omega Z_\epsilon(T) \ln (Z_\epsilon(T)) \, \dx\Big]
\end{multline*}
for all $m \le n$, where
\begin{equation*}
\eta (n) \to 0~\text{for}~ n \to \infty.
\end{equation*}
Thus we have proved
\begin{equation*}
\limsup_{\epsilon \to 0^+}  \int_0^T \psi_m \int_\Omega \phi_m ( Z_\epsilon^\gamma + \delta Z_\epsilon^\beta)Z_\epsilon \, \dxdt \le \int_0^T \int_\Omega \Ov{p} Z \, \dxdt, \quad \forall m \ge 1.
\end{equation*}
To conclude the proof of (\ref{pressurerelation}), we make use of a (slightly modified) Minty's trick. Since the nonlinearity $P(Z)= Z^\gamma + \delta Z^\beta $ is monotone, we have for any $v \in L^{\beta+1}(\QT)$
\begin{equation*}
\int_0^T \psi_m \int_\Omega \phi_m ( P(Z_\epsilon) - P(v))(Z_\epsilon -v) \, \dxdt \ge 0
\end{equation*}
and, consequently,
\begin{equation*}
\int_0^T \int_\Omega \Ov{p} Z \, \dxdt  + \int_0^T \psi_m \int_\Omega \phi_m P(v)v \, \dxdt - \int_0^T \psi_m  \int_\Omega \phi_m (\Ov{p}v +P(v)Z) \, \dxdt \ge 0.
\end{equation*}
Now, letting $m \to \infty $, we get
\begin{equation*}
\int_0^T \int_\Omega (\Ov{p}-P(v))(Z-v)\, \dxdt \ge 0
\end{equation*}
and the choice $v = Z+ \eta \varphi $, $\eta \to 0$, $ \varphi\in C^\infty_c(\QT) $ arbitrary, yields the desired conclusion
\begin{equation*}
\Ov{p} = Z^\gamma + \delta Z^\beta.
\end{equation*}

To finish the proof of
Proposition \ref{step2ex} we have to show (\ref{eststep28}). To this aim, we use as test function in (\ref{momstep2})
solution to (cf. Lemma \ref{Bogovskii} in Section 3)
$$
\dv \bm{\psi} = Z^\theta - \frac{1}{|\Omega|} \int_{\Omega} Z^{\theta} \dx
$$
with homogeneous Dirichlet boundary conditions, where $\theta > 0$ is a constant.  Due to properties of the Bogovskii operator we may show (the proof is similar to the case of compressible Navier--Stokes equations, see e.g. \cite{NoSt})
$$
\|Z\|_{L^{\gamma+\theta}(\QT)}^{\gamma+\theta}  + \delta \|Z\|^{\beta+\theta}_{L^{\beta+\theta}(\QT)} \leq C
$$
with $\theta \leq \min\{\frac{\gamma}{2}, \frac 23 \gamma -1\}$. Other estimates can be obtained easily.
The proof of Proposition \ref{step2ex} is finished.

\section{Passing to the limit in the artificial pressure term. Proof of Theorem \ref{mainthm2}}\label{s:7}
\label{artif_pressure_limit}
Our next goal is to let $\delta \to 0^+$. We will relax the assumptions on the growth of the pressure and on the regularity of the initial data.  We are again confronted with a missing estimate for the sequence of densities which would guarantee the strong convergence. Additional problems will arise from the fact that the a priori bounds for the density  do not allow us to apply the DiPerna--Lions transport theory, see Lemma \ref{ReSolLDP}. To overcome these difficulties, we will apply to system (\ref{wsstep2}) Feireisl's approach. Accordingly, the solution of  problem ($\ref{wsstep2}$) obtained in Proposition \ref{step2ex} above will be denoted $\vr_\delta,Z_\delta,\bu_\delta$.

\subsection{Limit passage based on a priori estimates}

Using estimates independent of the parameter $\delta$, i.e. (\ref{eststep21})--(\ref{eststep28}), as well as the procedure at the beginning of the previous section we show (see also \cite{NoSt})

\begin{subequations}
\begin{equation} \label{7.3a}
\vr_\delta \to \vr~\text{in}~C_{w}([0,T];L^\gamma (\Omega)),
\end{equation}
\begin{equation} \label{7.3b}
Z_\delta \to Z~\text{in}~C_{w}([0,T];L^\gamma (\Omega)),
\end{equation}
\begin{equation}\label{7.3c}
\bu_\delta \to \bu~\text{weakly in}~L^2(0,T;W^{1,2}_0(\Omega,\R^3)),
\end{equation}
\begin{equation} \label{7.3d}
\vr_\delta \bu_\delta \to \vr \bu ~\text{in}~C_{w}([0,T];L^{\frac {2\gamma}{\gamma+1}} (\Omega,\R^3)),
\end{equation}
\begin{equation} \label{7.3e}
\vr_\delta \bu_\delta \otimes \bu_\delta \to \vr \bu \otimes \bu ~\text{weakly in}~ L^q((0,T)\times \Omega, \R^{3\times 3}) \quad \text{for some}~q>1,
\end{equation}
\begin{equation} \label{7.3f}
\vr_\delta^\gamma   \to \Ov{\vr^\gamma} ~\text{weakly in}~L^{\frac {\gamma +\theta}{\gamma}}((0,T)\times \Omega),
\end{equation}
\begin{equation} \label{7.3g}
Z_\delta^\gamma   \to \Ov{Z^\gamma}  ~\text{weakly in}~L^{\frac {\gamma +\theta}{\gamma}}((0,T)\times \Omega),
\end{equation}
\begin{equation} \label{7.3h}
\delta Z_\delta^\beta   \to 0  ~\text{weakly in}~L^{q}((0,T)\times \Omega), \quad \text{for some}~q>1,
\end{equation}
\end{subequations}
passing to subsequences as the case may be.

Consequently, $\vr,Z, \bu$ satisfy
\begin{equation}\label{contstepfinal}
\partial_t \vr + \dv (\vr \bu) =0~\text{in}~\D'((0,T)\times \R^3) ,
\end{equation}
\begin{equation}\label{entstepfinal}
\partial_t Z + \dv (Z \bu) =0~\text{in}~\D'((0,T)\times \R^3) ,
\end{equation}
\begin{equation} \label{momstepfinal}
\partial_t(\vr \bu) + \dv (\vr \bu \otimes \bu) + \nabla \overline{Z^\gamma} = \mu \Delta \bu + (\mu+\lambda) \nabla \dv \bu~\text{in}~\D'(\QT,\R^3).
\end{equation}
Thus the only thing to complete the proof of Theorem \ref{mainthm2} is to show the strong convergence of $Z_\delta $ in $L^1(\QT)$ which is actually equivalent to identifying  $ \overline{Z^\gamma} = Z^\gamma $.

\subsection{Strong convergence of $Z_\delta$}

Recall that the cut-off functions $T$ and $T_k$ were introduced in (\ref{TX})--(\ref{TkX}).

\subsubsection{Effective viscous flux}
As in Section 6, we need the following auxiliary result:
\begin{lm}\label{effectivestep1}
Let $\vr_\delta,Z_\delta,\bu_\delta$ be the sequence of approximate solutions constructed by means of Proposition \ref{step2ex}. Then
\begin{equation} \label{7.11}
\lim_{\delta \to 0^+}  \int_0^T \psi \int_\Omega \phi \big( Z_\delta^\gamma -(\lambda+2\mu) \dv \bu_\delta\big) T_k(Z_\delta) \, \dxdt =   \int_0^T \psi \int_\Omega \phi  \big( \overline{Z^\gamma} -(\lambda+2\mu) \dv \bu\big) \overline{T_k(Z_\delta)} \, \dxdt
\end{equation}
for any $ \psi \in C_c^\infty((0,T)) $ and $\phi \in C_c^\infty(\Omega)$, passing to subsequences, if necessary.
\end{lm}
\begin{proof}
Recall that we have for $\delta>0$ the renormalized form of equation (\ref{entstep2})
\begin{equation} \label{8.2bbb}
\partial_t(T_k(Z_\delta))+\dv(T_k(Z_\delta)\bu_\delta)+(Z_\delta T'_k(Z_\delta)-T_k(Z_\delta))\dv\bu_\delta=0,
\end{equation}
however, for the limit we only have
\begin{equation} \label{8.3cc}
\partial_t(\Ov{T_k(Z)})+\dv(\Ov{T_k(Z)}\bu)+\Ov{(Z T'_k(Z)-T_k(Z))\dv\bu}=0,
\end{equation}
both in the sense of distributions.

We use as the test function  in the approximated momentum equation (\ref{momstep2})
the function
\[
 \bm{\varphi}_\delta= \psi \phi\nabla \Delta^{-1}[1_\Omega T_k(Z_\delta)] = \psi \phi \bm{{\cal A}} [1_\Omega T_k(Z_\delta)],\,k\in \tn{N},
\]
and for the limit equation (\ref{momstepfinal})
the test function
\[
 \bm{\varphi}=\psi \phi\nabla \Delta^{-1}[1_\Omega\Ov{T_k(Z)}] = \psi \phi \bm{{\cal A}} [1_\Omega\Ov{T_k(Z)}] ,\,k\in \tn{N}.
\]
Here, $\psi \in C^\infty_c(0,\infty)$ and $\phi \in C^\infty_c(\Omega)$, for the definition of $\bm{{\cal A}}$  see  Section 3.
Note that thanks to properties of $\psi$ and $\phi$ we indeed extend our domain from $\Omega$ onto the whole space $\R^3$. It allows then
to work with $\bm{{\cal A}}$ defined in terms of Fourier multipliers.

We get
\begin{equation} \label{ws27a}
\lim_{\delta \to 0^+} \int_0^T \psi \int_\Omega \Big(\phi Z_\delta^\gamma  T_k(Z_\delta) + Z_\delta^\gamma \nabla \phi \cdot \bm{{\cal A}} [1_\Omega T_k(\vr_\delta)]
\Big) \, \dxdt
\end{equation}
\[
- \lim_{\delta \to 0^+} \int_0^T \psi \int_\Omega \phi \Big( \mu  \nabla \bu_\delta: \bm{{\cal R}}[1_\Omega T_k(Z_\delta)] +
(\lambda + \mu) \dv \bu_\delta T_k(Z_\delta) \Big)  \, \dxdt
\]
\[
- \lim_{\delta \to 0^+} \int_0^T \psi \int_\Omega \Big( \mu  \nabla \bu_\delta \cdot \nabla \phi \cdot \bm{{\cal A}} [1_\Omega T_k(Z_\delta)]
+
(\lambda + \mu) \dv \bu_\delta \nabla \phi \cdot \bm{{\cal A}} [1_\Omega T_k(Z_\delta)] \Big)  \, \dxdt
\]
\[
= \int_0^T \psi \int_\Omega \Big(\phi \Ov{Z^\gamma} \ \Ov{T_k(Z)} - \Ov{Z^\gamma} \nabla \phi \cdot \bm{{\cal A}} [1_\Omega\Ov{T_k(Z)}]
\Big) \, \dxdt
\]
\[
- \int_0^T \psi \int_\Omega \phi \Big( \mu  \nabla \bu: \bm{{\cal R}}[1_\Omega \Ov{T_k(Z)}] +
(\lambda + \mu) \dv \bu \Ov{T_k(Z)} \Big)  \, \dxdt
\]
\[
-  \int_0^T \psi \int_\Omega \Big( \mu  \nabla \bu \cdot \nabla \phi \cdot \bm{{\cal A}}[1_\Omega \Ov{T_k(Z)}]
+
(\lambda + \mu) \dv \bu \nabla \phi \cdot \bm{{\cal A}}[1_\Omega \Ov{T_k(Z)}] \Big)  \, \dxdt
\]
\[
+ \lim_{\delta \to 0^+} \int_0^T \psi \int_\Omega \Big( \phi \vr_\delta \bu_\delta \cdot \bm{{\cal A}} [ \dv (T_k(Z_\delta)\bu_\delta) + (Z_\delta T_k'(Z_\delta)-T_k(Z_\delta))\dv \bu_\delta ] \\
\]
\[
-
\vr_\delta (\bu_\delta \otimes \bu_\delta) : \nabla \left( \phi \bm{{\cal A}}[1_\Omega T_k(Z_\delta)] \right) \Big)
  \, \dxdt
\]
\[
- \int_0^T \psi \int_\Omega \Big( \phi \vr \bu \cdot \bm{{\cal A}} [ \dv ( \Ov{T_k(Z)}\bu) + \Ov{(Z T_k'(Z)-T_k(Z))\dv\bu}  ]
\]
\[
- \vr (\bu \otimes \bu) : \nabla \left( \phi \bm{{\cal A}}[1_\Omega \Ov{T_k(Z)}] \right) \Big)
  \, \dxdt
\]
\[
-\lim_{\delta \to 0^+} \int_0^T \partial_t \psi \int_\Omega \phi \vr_\delta \bu_\delta\cdot \bm{{\cal A}} (T_k(Z_\delta)) \, \dxdt
 + \int_0^T \partial_t \psi \int_\Omega \phi \vr \bu\cdot \bm{{\cal A}} [\Ov{T_k(Z)}] \, \dxdt.
\]

We have
\[
\int_\Omega \phi \nabla \bu_\delta : \bm{{\cal R}} [1_\Omega T_k(Z_\delta)] \, \dx=
\int_\Omega \phi \sum_{i,j=1}^3 \left( \partial_{x_j} u^i_\delta {\cal R}_{ij} [1_\Omega T_k(Z_\delta)] \right)
\dx
\]
\[
= \int_\Omega \sum_{i,j=1}^3 \left( \partial_{x_j} (\phi u^i_\delta) {\cal R}_{ij} [1_\Omega T_k(Z_\delta)] \right)\, \dx\\
- \int_\Omega \sum_{i,j=1}^3 \left( \partial_{x_j} \phi  u^i_\delta {\cal R}_{ij} [1_\Omega T_k(Z_\delta)] \right)\dx
\]
\[
= \int_\Omega \phi \dv \bu_\delta T_k(Z_\delta)\, \dx+ \int_\Omega \nabla \phi \cdot \bu_\delta T_k(Z_\delta) \dx
- \int_\Omega \sum_{i,j=1}^3 \left( \partial_{x_j} \phi  u^i_\delta {\cal R}_{ij} [1_\Omega T_k(Z_\delta)] \right)\dx.
\]

Consequently, going back to (\ref{ws27a}) and dropping the compact terms, where we use
\[
{\cal A} [ 1_\Omega T_k(\vr_\delta) ] \to {\cal A} [1_\Omega \overline{T_k(\vr)}] \ \mbox{in}\ C([0,T] \times \Ov{\Omega}),
\]
we obtain
\begin{equation} \label{ws28a}
\lim_{\delta \to 0^+} \int_0^T \psi \int_\Omega \phi \Big( Z_\delta^\gamma T_k(Z_\delta) - (\lambda + 2 \mu) \dv \bu_\delta T_k(Z_\delta) \Big)
\, \dxdt
\end{equation}
\[
-\int_0^T \psi \int_\Omega \phi \Big( \Ov{Z^\gamma} \ \Ov{T_k(Z)} - (\lambda + 2 \mu) \dv \bu \Ov{T_k(Z)} \Big)
\, \dxdt
\]
\[
= \lim_{\delta \to 0^+} \int_0^T \psi \int_\Omega \Big( \vr_\delta \bu_\delta \cdot \bm{{\cal A}} [ \dv (T_k(Z_\delta) \bu_\delta)  ]
-
\vr_\delta (\bu_\delta \otimes \bu_\delta) : \bm{{\cal R}} [1_\Omega T_k(Z_\delta)] \Big)
  \, \dxdt
\]
\[
- \int_0^T \psi \int_\Omega \Big( \phi \vr \bu \cdot \bm{{\cal A}} [ \dv (\Ov{T_k(Z)} \bu) ] -
\vr (\bu \otimes \bu) : \bm{{\cal R}} [1_\Omega \Ov{T_k(Z)}]  \Big)
  \, \dxdt.
\]

Our goal is to show that the right-hand side of (\ref{ws28a}) vanishes. We write
\[
\int_{\Omega}\phi\Big[ \vr_\delta \bu_\delta \cdot \bm{{\cal A}} [ 1_{\Omega}\dv (T_k(Z_\delta) \bu_\delta) ] -
\vr_\delta (\bu_\delta \otimes \bu_\delta) : \bm{{\cal R}} [1_\Omega T_k(Z_\delta)]\Big]\dx
\]
\[
= \int_{\Omega} \phi \bu_\delta \cdot \Big[ T_{k}(Z_\delta) \bm{{\cal A}} [ \dv (1_{\Omega}\vr_\delta \bu_\delta) ] - \vr_\delta \bu_\delta \cdot \bm{{\cal R}} [1_\Omega T_k(Z_\delta)]\Big] \, \dx+ l.o.t.,
\]
where l.o.t. denotes lower order terms (with derivatives on $\phi$) and appear due to the integration by parts in the first term on the left-hand side.
We consider the bilinear form
\[
[\bm{v}, \bm{w}] = \sum_{i,j=1}^3 \Big( v^i \mathcal{R}_{ij} [w^j] - w^i \mathcal{R}_{ij} [v^j] \Big),
\]
where
\[
\bm{v} = \bm{v}(Z) = (T_k(Z), T_k(Z), T_k(Z)), \qquad \bm{w}= \bm{w}(\vr, \bu) = \vr \bu.
\]
We may write
\[
\sum_{i,j=1}^3 \Big( v^i \mathcal{R}_{ij} [w^j] - w^i \mathcal{R}_{ij} [v^j] \Big)
\]
\[
=\sum_{i,j=1}^3 \Big( (v^i - \mathcal{R}_{ij}[v^j] )  \mathcal{R}_{ij} [w^j] - ( w^i -
\mathcal{R}_{ij} [w^j] ) \mathcal{R}_{ij} [v^j] \Big)
= \bm{U} \cdot \bm{V} - \bm{W} \cdot \bm{Z},
\]
where
\[
U^i = \sum_{j=1}^3 (v^i - \mathcal{R}_{ij}[v^j] ),\
W^i = \sum_{j=1}^3 (w^i - \mathcal{R}_{ij}[w^j] ), \ \dv \bm{U} = \dv \bm{W} = 0,
\]
and
\[
V^i= \partial_{x_i} \left( \sum_{j=1}^3 \Delta^{-1} \partial_{x^j} w^j \right),\
Z^i= \partial_{x_i} \left( \sum_{j=1}^3 \Delta^{-1} \partial_{x^j} v^j \right),\ i=1,2,3.
\]

Therefore we may apply the Div-Curl lemma
(Lemma \ref{DivCurl}) and using
\[
T_k(Z_\delta) \to \Ov{T_k(Z)} \ \mbox{in}\  C_{\rm weak}([0,T]; L^q (\Omega)),\ 1\leq q<\infty,
\]
\[
\vr_\delta \bu_\delta \to \vr \bu \ \mbox{in}\ C_{\rm weak} ([0,T]; L^{2\gamma/(\gamma + 1)}(\Omega;\R^3)),
\]
we conclude that
\begin{equation}\label{ws29a}
T_k(Z_\delta) (t, \cdot) \bm{{\cal A}} [1_\Omega \dv(\vr_\delta \bu_\delta)(t,\cdot) ] - (\vr_\delta \bu_\delta)(t, \cdot) \cdot \bm{{\cal R}} [1_\Omega T_k(Z_\delta)(t, \cdot) ]
\end{equation}
\[
\to
\]
\[
\Ov{T_k(Z)}(t,\cdot) \bm{{\cal A}} [1_{\Omega} \dv (\vr \bu)(t, \cdot) ] - (\vr \bu)(t, \cdot) \cdot \bm{{\cal R}} [1_\Omega \Ov{T_k(Z)}(t, \cdot) ]
\]
\[
\mbox{weakly in}\ L^s(\Omega; \R^3) \ \mbox{for all}\ t \in [0,T],
\]
with
\[
s < \frac{2 \gamma}{\gamma + 1}.
\]
Note that $s> \frac{6}{5}$ since $\gamma > \frac 32$ and thus
the convergence in (\ref{ws29a}) takes place in the space
\[
L^q(0,T; W^{-1,2}(\Omega)) \ \mbox{for any}\ 1 \leq q < \infty;
\]
going back to (\ref{ws28a}), we have

\begin{equation}\label{ws30a}
\lim_{\delta \to 0^+} \int_0^T \psi \int_\Omega \phi \Big( Z_\delta^\gamma T_k(Z_\delta) - (\lambda + 2 \mu) \dv \bu_\delta T_k(Z_\delta) \Big)
\, \dxdt
\end{equation}
\[
=\int_0^T \psi \int_\Omega \phi \Big( \Ov{Z^\gamma} \ \Ov{T_k(Z)} - (\lambda + 2 \mu) \dv \bu \Ov{T_k(Z)} \Big)
\, \dxdt.
\]
\end{proof}

{
\begin{rmq}\label{gen_eff_visc_flux}
    Observe that an analogue of equality (\ref{8.2bbb}) holds also when we consider $\sigma_{\delta}$ instead of $T_k(Z_\delta)$, where $\sigma_\delta$ are uniformly essentially bounded and satisfy
    \[
        \partial_t \sigma_\delta + \dv(\sigma_\delta \bu_\delta) = f_\delta
    \]
    where $f_\delta$ are bounded in $L^2((0,T)\times\Omega)$ (see \cite{michalek} and \cite{plotsok}). This generalization will be necessary in Section \ref{equiv_formulations}.
\end{rmq}
}
\subsubsection{Oscillation defect measure and renormalized solutions}

The main results of this part are essentially taken over from \cite{FeCMUC}:
\begin{lm}\label{odm}
There exists a constant $c$ independent of $k$ such that
\begin{equation}
\limsup_{\delta \to 0^+} \| T_k(Z_\delta) - T_k(Z) \|_{L^{\gamma+1}(\QT)} \le c
\end{equation}
with $c$ independent of $k \ge 1$.
\end{lm}
\begin{proof}
One has
\begin{multline}\label{odmstep1}
\limsup_{\delta \to 0^+} \int_0^T \int_\Omega \Big(Z_\delta^\gamma T_k(Z_\delta) - \overline{Z^\gamma} \ \overline{T_k(Z)}\Big) \, \dxdt = \limsup_{\delta \to 0^+} \int_0^T \int_\Omega (Z_\delta^\gamma -Z^\gamma)(T_k(Z_\delta)-T_k(Z)) \, \dxdt \\
+ \int_0^T \int_\Omega (\overline{Z^\gamma} -Z^\gamma)( T_k(Z) -\overline{T_k(Z)}) \, \dxdt
\ge \limsup_{\delta \to 0^+} \int_0^T \int_\Omega (Z_\delta^\gamma -Z^\gamma)(T_k(Z_\delta)-T_k(Z)) \, \dxdt \\
\ge \limsup_{\delta \to 0^+}  \int_0^T \int_\Omega | T_k(Z_\delta)-T_k(Z) |^{\gamma+1} \, \dxdt,
\end{multline}
as $ Z \mapsto Z^\gamma$ is convex, $T_k$ concave on $ \R_+$, and
\begin{equation} \label{EST(i)}
(z^\gamma -y^\gamma)(T_k(z)-T_k(y)) \ge | T_k(z)-T_k(y) |^{\gamma+1}
\end{equation}
for all $z,y \ge 0$.
Hence,
\begin{equation} \label{EST(ii)}
\limsup_{\delta \to 0^+} \int_0^T \int_\Omega |T_k(Z_\delta)-T_k(Z)|^{\gamma +1} \, \dxdt \leq \int_0^T \int_\Omega (\Ov{Z^\gamma T_k(Z)}-\Ov{Z^\gamma} \ \Ov{T_k(Z)}) \, \dxdt.
\end{equation}
On the other hand,
\begin{multline}\label{odmstep2}
\limsup_{\delta \to 0^+} \int_0^T \int_\Omega \dv \bu_\delta \big(T_k(Z_\delta) - \overline{T_k(Z)}\big) \, \dxdt  \\
= \limsup_{\delta \to 0^+} \int_0^T \int_\Omega \big( T_k(Z_\delta) -T_k(Z) +T_k(Z) - \overline{T_k(Z)} \big) \dv \bu_\delta \, \dxdt \\
\le 2 \sup_{\delta>0} \| \dv \bu_\delta \|_{L^2(\QT)} \limsup_{\delta \to 0^+} \| T_k(Z_\delta) -T_k(Z) \|_{L^2(\QT)}.
\end{multline}
Relations (\ref{EST(i)}), (\ref{EST(ii)}) combined with Lemma \ref{effectivestep1} yield the desired conclusion.
\end{proof}

Using the result of Lemma \ref{odm} one has  the following crucial assertion (see Lemma \ref{ReSolFe}):

\begin{lm}\label{renormalizedfinal}
The limit functions $(Z,\bu)$  solve (\ref{content1}) in the sense of renormalized solutions, i.e.,
\begin{equation}
\partial_t b(Z) + \dv (b(Z) \bu) + ((b'(Z)Z-b(Z)) \dv \bu = 0
\end{equation}
holds in $\D'((0,T) \times \R^3)$ for any $b \in C^1(\R)$ satisfying (\ref{regb}) provided $(Z,\bu)$ are extended by
zero outside $\Omega$.
\end{lm}

\subsubsection{Strong convergence of the density}

We are going to complete the proof of Theorem \ref{mainthm2}. To this end, we introduce a
family of functions $(L_k)_{k \ge 1}$:
\begin{equation*}
L_k(z) =  z \int_1^z \frac{T_k(s)}{s^2} \ \rm{d }s .
\end{equation*}
Note that $L_k$ is convex for any $k \geq 1$ and
\begin{equation}\label{testbfinal}
Z L_k'(Z) -L_k(Z) =  T_k(Z).
\end{equation}
We can use the fact that $(Z_\delta,\bu_\delta)$ are renormalized solutions of (\ref{entstep2}) to deduce
\begin{equation}\label{strongcvstep10}
\partial_t L_k(Z_\delta) + \dv\big( L_k(Z_\delta) \bu_\delta\big) + T_k(Z_\delta) \dv \bu_\delta = 0
\end{equation}
in $ \D'((0,T)\times \R^3)$ with $Z_\delta$, $\bu_\delta$ extended by zero outside of $\Omega$.
Similarly, by virtue of (\ref{entstepfinal}) and Lemma \ref{renormalizedfinal} (as above, we may justify the use of $L_k(\cdot)$ by density argument)
\begin{equation}\label{strongcvstep11}
\partial_t L_k(Z) + \dv\big( L_k(Z) \bu\big) + T_k(Z) \dv \bu = 0
\end{equation}
in $ \D'((0,T)\times \R^3)$.

In view of (\ref{strongcvstep10}), we have
\begin{equation}\label{weakconvfinal}
L_k(Z_\delta) \to \overline{L_k(Z)}~\text{in}~C_w([0,T];L^q(\Omega))
\end{equation}
for all $1\leq q<\infty$. Hence (\ref{strongcvstep10}) yields
\begin{equation} \label{7.27}
\partial_t \Ov{L_k(Z)} + \dv(\Ov{ L_k(Z)} \bu) + \Ov{T_k(Z) \dv \bu} = 0
\end{equation}
in $ \D'((0,T)\times \R^3)$. Therefore, (\ref{strongcvstep11}) and (\ref{7.27}) imply
$$
\int_\Omega \Big(\Ov{L_k(Z(T))} - L_k(Z(T))\Big) \, \dx= \int_0^T \int_\Omega \Big(T_k(Z) \dv \bu - \Ov{T_k(Z) \dv\bu} \Big) \, \dxdt.
$$
Due to convexity  of $L_k(\cdot)$ we have
\begin{multline*}
0 \leq \int_0^T \int_\Omega \Big(T_k(Z) \dv \bu - \Ov{T_k(Z) \dv\bu} \Big) \, \dxdt   \\
\leq \int_0^T \int_\Omega \Big(T_k(Z)  - \Ov{T_k(Z)}\Big) \dv\bu \, \dxdt + \int_0^T \int_\Omega \Big(\Ov{T_k(Z)} \dv \bu - \Ov{T_k(Z) \dv\bu} \Big) \, \dxdt.
\end{multline*}
Now, the effective viscous flux equality (\ref{7.11}) and (\ref{EST(ii)}) imply
$$
 \frac{1}{2\mu +\lambda} \limsup_{\delta \to 0^+} \int_0^T \int_\Omega |T_k(Z_\delta)-T_k(Z)|^{\gamma +1} \, \dxdt\leq\int_0^T \int_\Omega \Big(\Ov{T_k(Z) \dv\bu} -\Ov{T_k(Z)} \dv \bu  \Big) \, \dxdt;
$$
whence
\begin{multline*}
\frac{1}{2\mu +\lambda} \limsup_{\delta \to 0^+} \int_0^T \int_\Omega |T_k(Z_\delta)-T_k(Z)|^{\gamma +1} \, \dxdt \leq \int_0^T \int_\Omega |T_k(Z)-\Ov{T_k(Z)}| |\dv \bu| \, \dxdt \\
\leq C \|T_k(Z) -\Ov{T_k(Z)}\|_{L^1(\QT)}^{\frac{\gamma-1}{2\gamma}} \|T_k(Z) -\Ov{T_k(Z)}\|_{L^{\gamma +1}(\QT)}^{\frac{\gamma+1}{2\gamma}}.
\end{multline*}
Recall that
$$
\|T_k(Z) -\Ov{T_k(Z)}\|_{L^1(\QT)} \leq \|T_k(Z) -Z\|_{L^1(\QT)} + \|\Ov{T_k(Z)}-Z\|_{L^1(\QT)},
$$
yielding
$$
\lim_{k\to \infty} \|T_k(Z) -\Ov{T_k(Z)}\|_{L^1(\QT)} = 0.
$$
As
$$
\sup_{k \geq 1} \limsup_{\delta \to 0^+} \|T_k(Z_\delta)-T_k(Z)\|_{L^{\gamma +1}(\QT)} <+\infty,
$$
we also have
$$
\lim_{k \to \infty} \limsup_{\delta \to 0^+} \|T_k(Z_\delta) -T_k(Z)\|_{L^{\gamma +1}(\QT)} = 0.
$$
Therefore, one verifies that
$$
Z_\delta \to Z ~\text{strongly in}~L^q(\QT)
$$
for any $q<\gamma +\theta$.
The proof of Theorem \ref{mainthm2}  is finished.

\section{Proof of equivalent formulations}\label{s:8}
\label{equiv_formulations}
From Theorem \ref{mainthm2} it follows that for any $\gamma >\frac 32$  there exists a triple of functions
\begin{equation} \label{8.1}
(\vr,Z,\bu) \in  L^\infty(0,T;L^\gamma(\Omega)) \times L^\infty(0,T; L^\gamma(\Omega)) \times L^2(0,T;W^{1,2}_0(\Omega,\R^3))
\end{equation}
satisfying equations (\ref{2a}) in the sense specified in Definition \ref{weaksolutionaux}. However, in what follows, we will use the result only for $\gamma \geq \frac 95$.

Our aim will be to deduce from this the existence of $s\in L^\infty((0,T)\times\Omega)$  such that the pressure in the momentum equation equals $p=\vr^\gamma{\cal T}(s)$
satisfying either equality (\ref{weak_Z}) or the distributional formulation of (\ref{ent_trans}) with corresponding initial data { in a similar way as suggested in Feireisl et al. \cite{FEKLNOZA}.}

\subsection{The case $\gamma\geq \frac{9}{5}$}
We first present the main ideas of the proof which corresponds to the situation
$$\frac{Z_0}{\vr_0}1_{\{\vr_0=0\}} = T(s_0)^{\frac 1\gamma}1_{\{\vr_0=0\}} =1.$$
Due to the construction we know that functions $(\vr,Z,\bu)$ extended by zero outside of $\Omega$ fulfill
equations (\ref{cont1}), (\ref{content1}) in the sense of distributions on the whole $(0,T)\times\R^3$.
Therefore, we may test both of these equations by $\xi_\eta(x- \cdot),$ where  $\xi_\eta$ is a standard mollifier. We obtain the following equations
\begin{equation}\label{vr_delta}
\partial_t\vr_\eta+\dv(\vr_\eta\bu)= r^1_\eta,
\end{equation}
\begin{equation} \label{8.3}
\partial_t Z_\eta+\dv(Z_\eta\bu)= r^2_\eta,
\end{equation}
satisfied a.e. in $(0,T)\times\R^3$, where by $a_\eta$ we denoted $a\ast\xi_\eta$.
From the Friedrichs lemma (see e.g. \cite{NoSt}) we know that $r^1_\eta$, $r^2_\eta$ converge to 0 strongly in $L^1((0,T)\times\R^3)$ as $\eta\to 0^+$ {(the strong convergence of $r^1_{\eta}$ requires the stronger assumption on $\gamma$).}
Now we multiply the first equation by $-\frac{(Z_\eta+\lambda)}{(\vr_\eta+\lambda)^2}$ and the second by $\frac{1}{\vr_\eta+\lambda}$ with $\lambda>0$, respectively.  Note that for $\eta$ fixed $\partial_t\vr_\eta$ and $\partial_t Z_\eta$ belong to $L^\infty(0,T; C_c^\infty(\R^3))$, so
these are sufficiently regular test functions. After some manipulations, we obtain the following equation
\[\partial_t \left( \frac{ Z_\eta + \lambda }{\vr_\eta + \lambda } \right) + \dv \left[ \left( \frac{ Z_\eta + \lambda }{\vr_\eta + \lambda } \right) \bu \right]
-\left[ \frac{( Z_\eta + \lambda) \vr_\eta }{(\vr_\eta + \lambda)^2} + \frac{\lambda}{ \vr_\eta + \lambda }\right] \dv \bu
\]
\[
= - r^1_\eta \frac{ Z_\eta + \lambda }{(\vr_\eta + \lambda)^2 } +r^2_\eta \frac{1}{\vr_\eta + \lambda}
\]
satisfied a.e. in $(0,T)\times\R^3$.
Note that $Z_\eta(t,x)=\int_{\R^3}Z(t,y)\xi_\eta(x-y){\rm d}y\leq c^\star\int_{\R^3}\vr(t,y)\xi_\eta(x-y){\rm d}y=c^\star\vr_\eta$, therefore
$$\frac{ Z_\eta + \lambda }{\vr_\eta + \lambda } \leq \frac{ c^\star\vr_\eta + \lambda }{\vr_\eta + \lambda }\leq \max\{1,c^\star\},\qquad \frac{1}{\vr_\eta+\lambda}\leq \frac{1}{\lambda}.$$
So, for $\lambda$ fixed, we may use the strong convergence of $\vr_\eta\to \vr$, $Z_\eta\to Z$ and the dominated convergence theorem to let $\eta\to 0$ and to obtain the following equation
\[
\partial_t \left( \frac{ Z + \lambda }{\vr+ \lambda } \right) + \dv \left[\left( \frac{ Z + \lambda }{\vr+ \lambda } \right) \bu \right]
- \left[ \frac{( Z + \lambda) \vr }{(\vr + \lambda)^2} + \frac{\lambda}{ \vr + \lambda }\right] \dv \bu=0
\]
which is satisfied  in the sense of distributions on $(0,T)\times\R^3$. Before we pass to the limit with $\lambda\to 0^+$ note that we may distinguish two situations
\begin{itemize}
\item for $\vr=0$ we have $Z=0$ and therefore $\frac{ Z + \lambda }{\vr + \lambda}=1$, while $ \frac{( Z + \lambda) \vr }{(\vr + \lambda)^2} + \frac{\lambda}{ \vr + \lambda }=1$,
\item for $\vr>0$ we have  $\frac{ Z + \lambda }{\vr + \lambda}\leq \max\{1,c^\star\}$, $\vr+\lambda\to\vr,\  Z+\lambda\to Z$ strongly in $L^\infty(0;T;L^2_{\rm{{loc}}}(\R^3))$, therefore $\frac{ Z + \lambda }{\vr + \lambda}\to \frac{Z}{\vr}$ strongly in $L^\infty((0,T)\times\R^3)$ and so $ \frac{( Z + \lambda) \vr }{(\vr + \lambda)^2} + \frac{\lambda}{ \vr + \lambda }\to \frac{Z}{\vr}$.
\end{itemize}

Recall that this construction corresponds to the choice $\zeta_0=1$ in (\ref{ICZETA}). In the more general case, for ${\cal T}(s_0)^{\frac 1\gamma} =A_0$ we would have to replace $\lambda$ in the numerator by $A_0 \lambda$.

{
 The case of non-constant $A_0 1_{\vr_0=0} ={\cal T}(s_0)^{\frac 1\gamma} 1_{\vr_0=0} \in L^{\infty}(\Omega)$ demands a bit more technical treatment. Let $A$ be any solution of the transport equation (\ref{ent_trans}) with the initial data $A_0$. Such solution can be found using smoothing of $\bu$ and solving the transport equation by the method of trajectories. Along with (\ref{vr_delta}) we also test the transport equation for $A$ by the same family of mollifiers obtaining
\[
    \partial_t A_{\eta} + \bu \nabla A_{\eta} = r^3_{\eta}
\]
with $r^3_{\eta} \to 0$ in $L^{1}((0,T)\times \R^3)$ as $\eta \to 0^+$. In combination with the continuity equation we obtain
\begin{equation}
    \partial_t (Z_\eta+\lambda A_{\eta})+\dv((Z_\eta+\lambda A_{\eta})\bu)= r^2_\eta+\lambda(r^3_\eta+A_\eta \dv \bu).
\end{equation}
We multiply the last equality by $\frac{1}{\vr_\eta +\lambda}$ and mimicking the previous approach we obtain
\[
    \partial_t \left(   \frac{Z_\eta+\lambda A_{\eta}}{\vr_\eta+\lambda}    \right)
    + \dv\left[     \left(   \frac{Z_\eta+\lambda A_{\eta}}{\vr_\eta+\lambda}    \right) \bu    \right]
-\left[ \frac{( Z_\eta + \lambda A_\eta) \vr_\eta }{(\vr_\eta + \lambda)^2} + \frac{\lambda A_\eta}{ \vr_\eta + \lambda }\right] \dv \bu
\]
\[
= - r^1_\eta \frac{ Z_\eta + \lambda A_\eta }{(\vr_\eta + \lambda)^2 } +r^2_\eta \frac{1}{\vr_\eta + \lambda}+r^3_\eta \frac{\lambda}{\vr_\eta+\lambda}.
\]
Next, we let $\eta \to 0^+$ and get
\[
    \partial_t \left(   \frac{Z+\lambda A}{\vr+\lambda}    \right)
    + \dv\left[     \left(   \frac{Z+\lambda A}{\vr+\lambda}    \right) \bu    \right]
-\left[ \frac{( Z + \lambda A) \vr }{(\vr + \lambda)^2} + \frac{\lambda A}{ \vr + \lambda }\right] \dv \bu=0.
\]
Let us denote $\theta = \frac{Z}{\vr}$ for $\vr>0$ and $\theta = A$ for $\vr=0$. Observe that $c_*\leq \theta \leq c^*$ almost everywhere in $(0,T)\times \Omega$. Once again, we use the uniform boundedness of $\frac{Z+A\lambda}{\vr + \lambda}$ and send $\lambda \to 0^+$ obtaining
\[
    \partial_t \theta + \dv(\theta \bu) - \theta \dv \bu = 0
\]
or, equivalently,
\begin{equation}\label{tr_theta}
\partial_t \theta+  \bu\cdot \nabla \theta=0
\end{equation}
in the sense of distributions on $(0,T)\times\R^3$. The initial condition $A_0$ is attained in the sense of weak solutions for the transport equation.
}
In addition,
we can renormalize this equation, using any $G\in C^1(\R)$ and deduce that
\begin{equation}\label{tr_theta_ren}
\partial_t G(\theta)+  \bu\cdot \nabla G(\theta)=0
\end{equation}
is also satisfied in the sense of distributions on $(0,T)\times\R^3$. Taking for example $G(\theta)={\cal T}^{-1}(\theta^\gamma)$, we obtain equation for $s$
\begin{equation} \label{8.7a}
\partial_t  s+ \bu \cdot \nabla s = 0,
\end{equation}
 and, for $G(\theta) = B({\cal T}^{-1}(\theta^\gamma))$, also its renormalized version
\begin{equation}\label{tr_s_ren}
\partial_t B(s)+  \bu\cdot \nabla B(s)=0
\end{equation}
satisfied in the sense of distributions on $(0,T)\times\R^3$ for any $B\in C^1(\R)$.

In order to obtain the weak solution to problem (\ref{content1_1}) we need to test equation (\ref{8.7a}) by $\vr$. This is, however, not allowed due to low regularity of $\vr$. Instead we will use $\varphi\vr_\eta$, where $\varphi\in C^\infty_c((0,T)\times \R^3)$ and $\vr_\eta$ satisfies (\ref{vr_delta}).
Here we essentially use the fact that $\vr\in L^2((0,T)\times\Omega)$, hence this step cannot be repeated for $\gamma$ less then $\frac{9}{5}$.
Then we also multiply (\ref{vr_delta}) by $\varphi s$ and sum up the obtained expressions to deduce
\[ \int_{0}^T \int_{\R^3}\vr_\eta s \partial_{t}\varphi\dxdt+\int_{0}^T \int_{\R^3} (\vr_\eta s \bu)\cdot \nabla\varphi\dxdt
=-\int_{0}^T \int_\Omega r^1_\eta s \varphi\dxdt.
\]
Having this formulation we pass to the limit with $\eta\to0^+$, note that the term on the
r.h.s. vanishes and therefore we obtain
\begin{equation} \label{8.8a}
\int_{0}^T \int_{\R^3}\vr s\partial_{t}\varphi\dxdt+\int_{0}^T \int_{\R^3} (\vr s\bu)\cdot \nabla\varphi\dxdt=0.
\end{equation}
Note that if we start from (\ref{tr_s_ren}), we can also get
\begin{equation} \label{dist_theta}
\int_{0}^T \int_{\R^3}\vr B(s)\partial_{t}\varphi\dxdt+\int_{0}^T \int_{\R^3} (\vr B(s)\bu)\cdot \nabla\varphi\dxdt=0
\end{equation}
for any $B\in C^1(\R)$.

Thus we have almost our formulation from Definition \ref{weaksolution}, except the initial condition. Indeed, for the moment we only know that equation $\partial_{t}(\vr s)+\dv(\vr  s\bu)=0$
is satisfied in the sense of distributions on $(0,T)\times \R^3$. Moreover, from the $L^\infty((0,T)\times\R^3)$ bound on $ s$ and the above equation, we deduce using Arzel\`a--Ascoli theorem  that $\vr s\in C_{w}([0,T];L^\gamma(\Omega))$.

To recover the initial and the terminal condition, we need to use a test function $\varphi$ from the space $C^1([0,T]\times\Ov{\Omega})$ instead of $C^\infty_c((0,T)\times\Omega)$. To this purpose we define the following function
\[
\varphi_\tau(t,x)=\left\{
\begin{array}{lll}
\frac{t}{\tau}\varphi(\tau,x) & \mbox{for} &t\leq \tau\\
\varphi(t,x) & \mbox{for} &\tau\leq t\leq T-\tau, \\
\frac{T-t}{\tau}\varphi(T-\tau,x) & \mbox{for} &T-\tau\leq t
\end{array}\right.
\]
for $\varphi\in C^1([0,T]\times\Ov{\Omega})$. Note that $\varphi_\tau$ is an admissible test function for (\ref{dist_theta}), we can write
\begin{equation} \label{dist_theta2}
\int_{\tau}^T \int_{\Omega}\vr s\partial_{t}\varphi\dxdt+\int_{0}^T \int_{\Omega} (\vr s\bu)\cdot \nabla\varphi\dxdt=-
\frac{1}{\tau}\int_{0}^\tau \int_{\Omega}\vr s\varphi(\tau,x)\dxdt+\frac{1}{\tau}\int_{T-\tau}^T \int_{\Omega}\vr s\varphi(T-\tau,x)\dxdt
.
\end{equation}
We represent function $\varphi(t,x)$ as $\varphi(t,x)=\psi(t)\zeta(x)$ (or approximate by such sums), where $\psi\in C^\infty_c((0,T))$, $\zeta\in C^\infty_c(\Ov{\Omega})$, then the r.h.s. of (\ref{dist_theta2}) equals
\[
-\frac{1}{\tau}\int_{0}^\tau \int_{\Omega}\vr s\varphi(\tau,x)\dxdt
+\frac{1}{\tau}\int_{T-\tau}^T \int_{\Omega}\vr s\varphi(T-\tau,x)\dxdt
\]
\[=
-\frac{\psi(\tau)}{\tau}\int_{0}^\tau \int_{\Omega}\vr s\zeta(x)\dxdt
+\frac{\psi(T-\tau)}{\tau}\int_{T-\tau}^T \int_{\Omega}\vr s\zeta(x)\dxdt,
\]
and by the weak continuity of $\vr s$, letting $\tau \to 0$, we conclude that
\begin{equation} \label{dist_theta3}
\int_{0}^T \int_{\Omega}\vr s\partial_{t}\varphi\dxdt+\int_{0}^T \int_{\Omega} (\vr s\bu)\cdot \nabla\varphi\dxdt= -\int_{\Omega}(\vr s)(0,\cdot)\varphi(0,\cdot)\dx+\int_{\Omega}(\vr s)(T,\cdot)\varphi(T,\cdot)\dx
\end{equation}
\[=-\int_{\Omega}S_0(\cdot)\varphi(0,\cdot)\dx+\int_{\Omega}(\vr s)(T,\cdot)\varphi(T,\cdot)\dx\]
and so the statement of Theorem \ref{mainthm1} is proven. Similarly we may get the initial condition for $s(t,\cdot)$.

\subsection{The case $\gamma>\frac{3}{2}$}
The case of general $\gamma$ has to be treated differently, due to the lack of $L^2((0,T)\times\Omega)$ estimate on $\vr$ and $Z$. The latter is necessary to apply the DiPerna-Lions technique of renormalization of the transport equation \cite{diperna1989ordinary}. In the general case, we have to use more subtle technique developed by Feireisl, see e.g. \cite{feireisl2004dynamics} and used recently in \cite{michalek} to study stability of solutions to system (\ref{1e}).
{ In this section we will extend the stability result and prove existence of solutions to system (\ref{1e}) by giving a suitable sequence of approximative problems.
}

As a starting point for the further analysis we take system  (\ref{wsstep2})  with $\beta>\max\{\gamma,4\}$ and initial data $Z_{0,\delta}=\frac{\rho_{0,\delta}}{\zeta_{0,\delta}}$, with $\zeta_{0,\delta}$ satisfying \eqref{ICZETA}.  At this stage, we are able to repeat the procedure described in the previous section in order to recover equation (\ref{tr_theta}) for $\theta_\delta$ and its renormalized version (\ref{tr_theta_ren})
in the sense of distributions on $(0,T)\times\Omega$. Moreover, $ \theta_\delta, \theta_\delta^{-1}$ are bounded in $L^\infty((0,T)\times\Omega)$ uniformly with respect to $\delta$.
Thus, our system
\begin{subequations}\label{wsstep3a}
\begin{equation}\label{contstep3a}
\partial_t \vr_\delta + \dv(\vr_\delta \bu_\delta)= 0,
\end{equation}
\begin{equation}\label{entstep3a}
\partial_t B(\theta_\delta) +  \bu_\delta\cdot \nabla B(\theta_\delta)=0,
\end{equation}
\begin{equation}\label{momstep3a}
\partial_t (\vr_\delta \bu_\delta) + \dv(\vr_\delta \bu_\delta \otimes \bu_\delta) + \nabla (\vr_\delta \theta_\delta)^\gamma + \delta  \nabla (\vr_\delta \theta_\delta)^\beta = \dv\tn{S}(\nabla \bu_\delta),
\end{equation}
\end{subequations}
where $\theta_\delta = \frac{Z_\delta}{\vr_\delta}$, is satisfied in the sense of distributions on $(0,T)\times\Omega$. { Observe that $\vr_{\delta}$ belongs (not necessarily uniformly with respect to $\delta$) to $L^{\beta}((0,T)\times\Omega)$ for each $\delta>0$. At this stage we can use the stability result given by Theorem 3.1 in \cite{michalek} and finish the proof of Theorem \ref{mainthm3}.}

For the sake of completeness, we will provide the limit process $\delta \to 0^+$ following the arguments from \cite{michalek}.
We take $\zeta_\delta= \theta^{-1}_\delta$ and denote $\zeta$ the weak$-\star$ limit of $\zeta_\delta$ (or its subsequence) in $L^{\infty}((0,T)\times \Omega)$. For any $\delta>0$ the pair $(\zeta_\delta,\bu_\delta)$ satisfies the transport equation in the weak sense (see Definition \ref{weaksolution_without_rho}) along with the initial data $\zeta_{0,\delta}=\frac{Z_{0,\delta}}{\vr_{0,\delta}}$.
As we know from Section \ref{artif_pressure_limit}, sequence $Z_\delta = \frac{\rho_\delta}{\zeta_\delta}$ (or its subsequence) converges strongly in $L^q((0,T)\times \Omega)$ to $Z$ for any $q<\gamma+\theta$.
Hence for the same $q$ we have
\[
    \vr_\delta = Z_\delta \zeta_\delta \to Z \zeta ~\text{weakly in } L^q((0,T)\times \Omega).
\]
Therefore $\zeta$, $\vr$ and $\bu$ satisfy in the weak sense
\begin{subequations}\label{wsstep4}
\begin{equation}\label{contstep4}
\partial_t \vr + \dv(\vr \bu)= 0,
\end{equation}
\begin{equation}\label{momstep4}
\partial_t (\vr \bu) + \dv(\vr \bu \otimes \bu) + \nabla \left(\frac{\vr}{\zeta}\right)^\gamma  = \dv\tn{S}(\nabla \bu).
\end{equation}

The next step is to show that the pair $(\zeta,\bu)$ satisfies the transport equation
\begin{equation}\label{transp_eq4}
    \partial_t \zeta -\bu \cdot \nabla \zeta = 0
\end{equation}
in the weak sense.
\end{subequations}
We apply the Div-Curl lemma (Lemma \ref{DivCurl}) with
\[
    \bU_\delta=(\zeta_\delta,\zeta_\delta \bu_\delta),\quad \bV_\delta=(\bu^j_\delta,0,0,0),
\]
where $j\in \{1,2,3\}$. We know that $\dv \bU_{\delta}$ and $\operatorname{curl} \bV_\delta$ are bounded in $L^2((0,T)\times \Omega)$, hence precompact in $W^{-1,2}((0,T)\times \Omega)$. Therefore we obtain $\zeta_\delta \bu_\delta \to \zeta \bu$ weakly in $L^2((0,T)\times \Omega,\R^3)$.
Due to the strong convergence of the pressure terms $Z^\gamma_\delta$ we get by the means of Lemma \ref{effectivestep1} (and Remark \ref{gen_eff_visc_flux})
\[
    \zeta_\delta \dv \bu_\delta \to \zeta \dv \bu ~\text{weakly in } L^2((0,T)\times\Omega).
\]
Therefore (\ref{transp_eq4}) is satisfied in the weak sense and due to the boundedness of $\zeta$ it is also a renormalized solution.
The proof of Theorem \ref{mainthm3} is complete.

\bigskip

\noindent{\bf Acknowledgements.}
The work  of D.M. and A.N. has been partially supported by the MODTERCOM project within the APEX programme of the Provence-Alpes-C\^ote d'Azur region.
The work of P.B.M. and  E.Z. has been partly supported by the National  Science  Centre  grant 2014/14/M/ST1/00108 (Harmonia).
M.M. acknowledges the support of the GACR (Czech Science Foundation) project 13-00522S in the framework of RVO: 67985840.
The work of M.P. was partly supported by the GACR (Czech Science Foundation) project 16-03230S.

\end{document}